\documentclass{article}

\usepackage{arxiv}

\usepackage[utf8]{inputenc} 
\usepackage[T1]{fontenc}    
\usepackage{hyperref}       
\usepackage{url}            
\usepackage{booktabs}       
\usepackage{amsfonts}       
\usepackage{nicefrac}       
\usepackage{microtype}      
\usepackage{lipsum}		
\usepackage{graphicx}
\usepackage{doi}

\usepackage{graphicx}
\usepackage{amssymb}
\usepackage{mathtools}
\usepackage{tikz}
\usepackage{float}
\usepackage{tasks}
\usepackage{stix}
\usepackage{booktabs}
\usepackage{subcaption}
\usepackage{csquotes}
\usepackage{comment}
\usepackage{amsthm}
\usepackage{amsmath}
\usepackage{subcaption}
\usepackage{proof}
\usetikzlibrary{positioning}
\usetikzlibrary{shapes.geometric}
\pgfmathtruncatemacro\distance{1}

\newtheorem{theorem}{Theorem}[section]
\newtheorem{lemma}[theorem]{Lemma}

\newtheorem{proposition}[theorem]{Proposition}
\newtheorem{remark}[theorem]{Remark}
\newtheorem{definition}[theorem]{Definition}
\newtheorem{example}[theorem]{Example}

\newcommand{\conn}{{\copyright}}

\newcommand{\nats}                  {{\mathbb N}}

\newcommand{\Fm}{\mathbf{Fm}}

\newcommand{\cL}{\mathcal{L}}

\newcommand{\Quo}{\mathsf{Hom}_s^\der}
\newcommand{\QuoMu}{\mathsf{Hom}_s^{\der_{\Mt_1}}}
\newcommand{\QuoMd}{\mathsf{Hom}_s^{\der_{\Mt_2}}}

\newcommand{\Rexp}{\mathsf{Hom}^{-1}_s}
\newcommand{\Up}{\mathsf{P}_\Uf}
\newcommand{\SubN}{\mathsf{S}}
\newcommand{\Nmatr}{\mathsf{Nmatr}}
\newcommand{\Matr}{\mathsf{matr}}

\newcommand{\cM}{\mathcal{M}}
\newcommand{\Mt}{\mathbb{M}}

\newcommand{\der}    						  {\vartriangleright}

\newcommand{\tuple}[1]                         {{\langle #1\rangle}}

\newcommand{\Uf}                                 {{\mathcal U}}

\newcommand{\bA}                                 {{\mathbf A}}

 \newcommand{\sub}{\mathsf{sub}}
\newcommand{\var}{\mathsf{var}}

\newcommand{\Ax}             {\mathsf{Ax}}

\newcommand{\Val}{\textrm{Val}}

\newcommand{\TCon}{\textrm{T}}

\newcommand{\botop}{\ensuremath{\bot\mkern-14mu\top}}

\title{Some more theorems on structural entailment relations and non-deterministic semantics}

\author{{\hspace{1mm}Carlos Caleiro} \\
	{SQIG -- Instituto de Telecomunica\c c\~oes}\\
{Departamento de Matem\'atica -- Instituto Superior T\'ecnico}\\
{Universidade de Lisboa, Portugal}\\
	\texttt{ccal@math.tecnico.ulisboa.pt} \\
	\And
	\href{https://orcid.org/0000-0000-0000-0000}{\hspace{1mm}S\'ergio Marcelino} \\
	{SQIG -- Instituto de Telecomunica\c c\~oes}\\
{Departamento de Matem\'atica -- Instituto Superior T\'ecnico}\\
{Universidade de Lisboa, Portugal}\\
	\texttt{smarcel@math.tecnico.ulisboa.pt} \\
	 \AND
	  Umberto Rivieccio \\
{Departamento de L\'ogica, Historia y Filosof\'ia de la Ciencia}\\
{Universidad Nacional de Educaci\'on a Distancia} \\
{Madrid, Spain} \\
 \texttt{umberto@fsof.uned.es}
}

\date{}

\renewcommand{\headeright}{ }
\renewcommand{\undertitle}{Preprint
}
\renewcommand{\shorttitle}{Some more theorems on  structural entailment relations and non-deterministic semantics}

\hypersetup{
pdftitle={Some more theorems on structural entailment relations and non-deterministic semantics},
pdfsubject={q-bio.NC, q-bio.QM},
pdfauthor={C Caleiro, S Marcelino, U Rivieccio},
pdfkeywords={Entailment relations, Multiple-conclusion logic, Matrix models, Non-deterministic matrices, Abstract algebraic logic},
}

\begin{document}
\maketitle

\begin{abstract}
We extend classical work by
Janusz Czelakowski on
the closure properties of the class of matrix models of 
entailment relations
-- nowadays more commonly called \emph{multiple-conclusion
logics} --
to the setting of non-deterministic matrices (\emph{Nmatrices}),
      characterizing  
     the
    Nmatrix models of an arbitrary logic through a generalization of the standard class operators to the non-deterministic setting. 
    We highlight the main differences that appear in this more general setting, in particular:
    the possibility to obtain Nmatrix quotients using any compatible equivalence relation (not necessarily a congruence);
    the problem of determining when strict homomorphisms preserve the logic of a given Nmatrix;
    the fact that the operations of taking images and preimages cannot be swapped, which
    determines the exact sequence of operators that generates, from any complete semantics, 
    the  class of all Nmatrix models of a logic.    
   Many results, on the other hand, generalize smoothly to the non-deterministic setting: 
    we show for instance that a logic is finitely based if and only if both the class of its Nmatrix models and its complement are closed under ultraproducts.
    We conclude by mentioning possible developments in adapting the  Abstract Algebraic Logic approach to logics induced by Nmatrices 
    and the associated equational reasoning over non-deterministic algebras.
    \end{abstract}

\keywords{Entailment relations\and Multiple-conclusion logic\and Matrix models\and Non-deterministic matrices\and Abstract algebraic logic }

\section{Introduction}

Over the course of the one hundred years
elapsed since the work of A. Tarski, a \emph{logic} has been often conceived
as a \emph{consequence relation} 
$\vdash\ \subseteq \wp(Fm)\times Fm$ 
satisfying certain properties (cf.~Subsection~\ref{ss:smcl}),
the main objects of interest 
 being pairs $\langle \Gamma, \varphi \rangle$ such that $\Gamma$ is a set of formulas
 (the \emph{premisses} of an argument,  derivation or rule) 
and $\varphi$ a single formula (the \emph{conclusion}). 
This view, which has become  standard in algebraic logic,  is by no means the only possible one,
the study of logics as sequent systems 
being perhaps the most well-known and time-honoured alternative. 

A \emph{sequent} is usually defined as a pair $\langle \Gamma, \Delta \rangle$ such that $\Gamma$ and $\Delta$
are both sets (or sequences, or multisets) of formulas. In a sequent system
one may  not only express the fact that the formulas  in $\Delta$ follow (or are derivable) from those in $\Gamma$, but also that a certain sequent $\langle \Gamma, \Delta \rangle$ follows (or is derivable) from a set of sequents $\{ \langle \Gamma_1, \Delta_1 \rangle, \ldots \langle \Gamma_n, \Delta_n \rangle \}$.

A third and heretofore  less explored alternative
(which can be traced back to~\cite{Scott})
is the view of logics as  \emph{entailment relations}; in the present paper, following
the tradition initiated in the book~\cite{ShoesmithSmiley},
we shall call them
\emph{multiple-conclusion logics}.
This  is somewhat a compromise between the preceding two alternatives,
for a logic is now regarded  as a relation $\der\ \subseteq \wp(Fm)\times \wp(Fm)$.
The main objects of interest are thus sequents $\langle \Gamma, \Delta \rangle $, with $\Gamma$ and $\Delta$
 both sets of formulas, but the focus is on valid sequents only: the formalism does not allow one to speak of derivations among sequents. Ordinary  consequence relations -- 
which, by contrast,   we are going to call \emph{single-conclusion logics} --  may  be retrieved in this setting as the special case
 where the set of conclusions $\Delta$ is required to be a singleton.
 
The systematic investigation of 
multiple-conclusion logics
was inaugurated  in  1978 with the book by D.~Shoesmith and T.~Smiley~\cite{ShoesmithSmiley}. 
This new approach soon attracted the interest of prominent logicians from the Polish school,
who proceeded to apply and extend their techniques 
(which had been developed for the
single-conclusion case
by J.~{\L}os, 
R.~Suszko and R.~W\'ojcicki) 
to the more general multiple-conclusion setting.  
This endeavour  resulted in a few significant contributions, notably two papers by
J.~Zygmunt~\cite{Zygmunt1979} 
and J.~Czelakowski~\cite{czela1983}, but  research on  
 multiple-conclusion  logics
appears to have been  otherwise relatively dormant in subsequent decades
(however, see e.g.~\cite{smiley1996,rumfitt2000,beall2000}), 
until a  recent revival: see e.g.~\cite{wollic19,synt,marcelino2022,EPTCS357,newfibring,rinaldi2019,ripley2021,chemla2021}.

Czelakowski's paper~\cite{czela1983}, which is the main inspiration for our present research, focused in particular on
the closure properties of the class of matrix models of a 
multiple-conclusion
logic\footnote{Czelakowski's results are reviewed in Section 3 of J.M.~Font and R.~Jansana's
chapter included in the present book.}.
Formally, a \emph{(logical) matrix} is a fairly simple first-order structure consisting of a pair $\tuple{\bA , D}$ where
$\bA$ is an algebra and
 $D\subseteq A$  a subset of designated elements. Matrices have been considered  at least since R.~W\'ojcicki's work~\cite{Woj}  as the standard algebra-based models of (single-conclusion) logics. 

The fact that a matrix is a little more than an ordinary algebra allows for a smooth extension
of many universal algebraic notions, such as subalgebra, homomorphism and (ultra)product. 
Considering the associated class operators,
one can then investigate the relationship among the closure properties of the class $\Matr(\mathcal{L})$ of all matrix models 
of a given 
(single- or multiple-conclusion)
logic $\mathcal{L}$, 
the type of first-order
sentences that axiomatize $\Matr(\mathcal{L})$ and  logical properties of $\mathcal{L}$ itself (structurality, compactness, etc.).
This kind of questions is addressed in the paper~\cite{czela1983}: in the present work, proceeding in parallel with Czelakowski's, we shall endeavour to extend the investigation to  the relationship between logical properties of a given  
multiple-conclusion logic
$\der$
and  closure properties of the corresponding class $\Nmatr(\der)$ of its \emph{non-deterministic} matrix models.

A \emph{non-deterministic matrix} (\emph{Nmatrix}) is a pair $\tuple{\bA , D}$ which differs from an ordinary logical matrix
in that $\bA$ is a \emph{multialgebra} rather than a standard algebra.
That is, $\bA$ consists of a set $A$ equipped with non-deterministic algebraic operations, each of them  being a function
of type  
$A^k\to \wp(A)\setminus \{\emptyset\}$
which associates a non-empty set of possible outputs to every $k$-tuple of elements from $A$
(see~\cite{Ma1934,Gr1962,No79PhDThesis,cirulis2003multi,CoLe2003,WaMe2018,Golzio18}
for different frameworks in which the theory of multialgebras has been studied). For a more formal definition
of multialgebra, see Section~\ref{sec:prelim}.

The idea of considering non-deterministic semantics in logic can be traced back at least to~\cite{quine,ivlev,creth}, but its systematization in the notion of  Nmatrix began in~\cite{Avron,avlev05}. The initial motivation 
originated from the analysis of 
modal and paraconsistent  
logics, and was mainly related to the compression power afforded by  non-determinism, 
which allows one, for instance,
to characterize, by a single finite Nmatrix, logics that 
are not complete with respect to any finite set of 
finite deterministic matrices. {The simplest of such examples is the smallest logic in any given signature, axiomatized by the empty set of rules (see Example~\ref{ex2val}).}  
Non-determinism moreover provides a number of modular  and often constructive bridges between semantics and analytic proof-calculi; for further information  
we refer the reader to~\cite{avron2005,Avron2012atLICS,AvronBK07,Baaz2013,OnAxRexp,newfibring}.

It is well known that the logic defined by a finite deterministic matrix
enjoys certain properties -- besides the one known as \emph{cancellation}~\cite{SS1971}, it is also compact and
locally tabular~\cite{MR2017};  the two latter in fact also hold for any logic  determined
by a finite set of finite matrices~\cite[Thm.~3.15]{finval}.
When considering Nmatrices, the compactness result still applies~\cite{avlev05},
but  cancellation and local tabularity may fail (see
Remark~\ref{rem:noprops}).

As soon as  we consider the question of extending Czelakowski's~\cite{czela1983} results to the setting of Nmatrices,
we are immediately faced with the problem that there is no standard notion of a \emph{reduced model}. 
Indeed, it makes perfect sense to factor an Nmatrix
$\tuple{\bA , D}$ by any equivalence relation compatible with $D$,
obtaining  a quotient that  will define, in general, a weaker logic:
 this is in sharp contrast to all we know about the role of congruences and the Leibniz operator on deterministic matrices,
 for the quotient of a matrix by an equivalence relation
 is not guaranteed to be a deterministic matrix.
 In this work we shall see that, despite these difficulties, one can still recover (adapted versions of) some of the nice results of matrix semantics in this more general setting.

The rest of the paper is organized as follows. Section~\ref{sec:prelim} recalls the basics on multiple-conclusion logics and Nmatrices, and illustrates their use. Section~\ref{sec:nmatrices} introduces and studies Nmatrix homomorphisms, their properties, and the corresponding constructions. In Section~\ref{sec:results} we prove our main results, characterizing compact and finitely based multiple-conclusion logics using ultraproducts of Nmatrices, and provide a general characterization of the Nmatrix models of a multiple-conclusion logic. We conclude, in Section~\ref{sec:conc}, with a discussion of the results obtained and an outline of future work.

\section{Preliminaries}\label{sec:prelim}

\paragraph{Propositional languages, algebras and multialgebras.}
A \emph{signature}\index{signature} is an algebraic similarity type, i.e.~a family of sets of connectives $\Sigma=\{\Sigma^{k}:k<\omega\}$ indexed by arity. 
A \emph{$\Sigma$-multialgebra} \index{multialgebra} is a tuple $\bA=\tuple{A,\cdot_\bA}$ such that 
$\conn_\bA:A^k\to \wp(A)\setminus \{\emptyset\}$
for each $\conn\in \Sigma^{k}$.
The notion of ordinary algebra is recovered 
by allowing $\conn_\bA$ to output only singleton sets.
In this paper we shall identify algebras with this particular case  
of (deterministic) multialgebras.

A \emph{homomorphism} \index{multialgebra homomorphism} between multialgebras\footnote{For the sake of readability, we write 
$\bA_1=\tuple{A_1,\cdot_1}$ instead of $\bA_1=\tuple{A_1,\cdot_{\bA_1}}$, etc.}
$\bA_1=\tuple{A_1,\cdot_1}$
and $\bA_2=\tuple{A_2,\cdot_2}$ 
 is a function $h:A_1\to A_2$ such that $h(\conn_1(a_1,\ldots,a_k))\subseteq \conn_2(h(a_1),\ldots,h(a_k))$
 for each  $\conn\in \Sigma^{k}$ and $a_1,\ldots,a_k\in A_1$.
 When $\bA_1$ and $\bA_2$ are algebras, this gives us
 the usual notion of algebraic homomorphism. We denote by $\mathsf{hom}(\bA_1,\bA_2)$  the set of multialgebra homomorphisms between $\bA_1$ and $\bA_2$.
 
 Throughout the paper we will consider a fixed countable signature $\Sigma$ and
 a denumerable set of propositional variables $P$,  denoting by 
$Fm$
(also commonly denoted by $L_\Sigma(P)$)
the language containing the formulas built from $\Sigma$ and $P$ in the usual way. 
  $\Fm=\tuple{Fm,\cdot_\Fm}$    may be seen as
  the absolutely free $\Sigma$-algebra whose interpretation 
   of connectives $\conn\in \Sigma$ is the function
   $\conn_\Fm(\varphi_1,\ldots,\varphi_k)=\conn(\varphi_1,\ldots,\varphi_k)$
    yielding the more complex formula built from the connective and the input subformulas.
    When viewing $\Fm=\tuple{Fm,\cdot_\Fm}$
    as a multialgebra, one has  $\conn_\Fm(\varphi_1,\ldots,\varphi_k)=\{\conn(\varphi_1,\ldots,\varphi_k)\}$.  
A \emph{substitution} is a function $\sigma: P\to Fm$ that is extended to an endomorphism
$\cdot^\sigma:Fm\to Fm$ in the usual way. As usual, we let $\sub(\Gamma)$ denote the set of all subformulas of formulas in $\Gamma\subseteq Fm$.
Given $\varphi\in Fm$, we denote by $\mathsf{depth}(\varphi)$ the depth of 
$\varphi$, defined inductively by $\mathsf{depth}(\varphi)=0$ if $\varphi\in P$,
and $\mathsf{depth}(\conn(\varphi_1,\ldots,\varphi_k))=1+\mathsf{max}(\{\mathsf{depth}(\varphi_i):1\leq i\leq k\})$. We assume that $\mathsf{max}(\emptyset)=0$, so that $\mathsf{depth}(\conn)=1$ in case $\conn\in\Sigma^0$.

\paragraph{Single-conclusion and multiple-conclusion logics.}
\label{ss:smcl}
A \emph{multiple-conclusion logic} \index{multiple-conclusion logic} on $Fm$ is a relation 
$\der\ \subseteq \wp(Fm)\times \wp(Fm)$ satisfying the  properties (O), (D), (C) and (S) listed below, for every $\Gamma,\Delta,\Gamma',\Delta'\subseteq Fm$.

\begin{itemize}
 \item[(O)] 
 
 If $\Gamma\cap \Delta\neq \emptyset$ then $\Gamma\der \Delta$. 
 
  \item[(D)] 
    
    If $\Gamma\der \Delta$ then $\Gamma,\Gamma'\der \Delta,\Delta'$. 

  \item[(C)] 
  
  If $\Gamma,\Omega\der \overline{\Omega},\Delta$ for each $\Omega\subseteq Fm$ 
  (where $\overline{\Omega}=Fm\setminus {\Omega}$), then $\Gamma\der \Delta$. 
  
  
 \item[(S)] If $\Gamma\der \Delta$ then $\Gamma^\sigma\der \Delta^\sigma$ for each substitution $\sigma:P\to Fm$.  

\end{itemize}
Property (C) is usually known as \emph{cut for sets} or \emph{transitivity}, but we prefer to call it \emph{case exhaustion}.
(O) is usually known as \emph{overlap}  or \emph{reflexivity}, (D) as \emph{dilution}  or \emph{monotonicity}, and 
(S) a \emph{substitution invariance} or \emph{structurality} (see~\cite{ShoesmithSmiley,Woj,Scott}). 
{
Note that in (D), (C) and  throughout the paper we adopt the convention of writing $\Gamma, \Gamma'$ instead of
$\Gamma \cup \Gamma'$, etc.}

The relation $\der$ is \emph{compact}\index{compact} if it further satisfies the property (F) below for every $\Gamma\subseteq Fm$:

\begin{itemize}
 \item[(F)]  
if $\Gamma\der \Delta$ then there exist finite sets $\Gamma_0\subseteq \Gamma$ and
$\Delta_0\subseteq \Delta$ such that $\Gamma_0\der \Delta_0$.
\end{itemize}

\noindent
{ 
The  standard properties of a single-conclusion logic  \index{single-conclusion logic} can be recovered 
by  
specializing 
the properties 
(O), (D), (S) and (F) by  allowing only singleton sets for conclusions, letting $\Delta=\Delta_0$ and $\Delta'=\emptyset$.} However, property (C) needs a different formulation (C${}^\TCon$), applicable to every $\Gamma,\Delta\subseteq Fm$ and $\varphi\in Fm$:

\begin{itemize}
 \item[(C${}^\TCon$)]   
if $\Gamma,\Delta\der \varphi$ and $\Gamma\der \psi$ for each $\psi\in\Delta$, then $\Gamma\der \varphi$.
\end{itemize}

In order to avoid  confusion, we  shall use the symbol   $\vdash$ for single-conclusion logics. 
Given a  multiple-conclusion logic $\der$, we denote by  
 $\vdash_\der\subseteq  \wp(Fm)\times Fm$ 
the \emph{single-conclusion companion}  \index{single-conclusion companion} of $\der$
defined by   
{
$\Gamma\vdash_\der \varphi$ if and only if
$\Gamma\der \{\varphi\}$.}
 In general there may be different multiple-conclusion logics sharing the same single-conclusion companion~\cite[Ch.~5]{ShoesmithSmiley}.

\paragraph*{Finitely based logics.}
Given a set $R\subseteq \wp(Fm)\times\wp(Fm)$ of rules, let $\der_R$ be the smallest multiple-conclusion logic such that ${R}\,{\subseteq}\der_R$. 
We say that a rule $\frac{\Gamma}{\Delta}\in \wp(Fm)\times\wp(Fm)$ is \emph{finite }
when $\Gamma$ and $\Delta$ are finite.
We further say that a multiple-conclusion logic $\der$ is \emph{finitely based} \index{finitely based} whenever
there is a finite set of finite rules $R$ 
such that $\der{=}\der_R$.\smallskip

To any single-conclusion logic $\vdash$ we can associate a multiple-conclusion logic 
by setting $\Gamma \der_\vdash \Delta$ iff $\Gamma\vdash \varphi$ for some $\varphi\in \Delta$.
Then  $ \der_\vdash$ is the least multiple-conclusion logic \index{least multiple-conclusion companion}  having $\vdash$ as single-conclusion companion,
and mirrors certain  properties of $\vdash$ (in particular, $\vdash$ is finitely axiomatizable using single-conclusion rules whenever $\der_\vdash$ is finitely based).

\subsection{Non-deterministic matrix semantics}  
\label{ss:ndms} 

\begin{definition}
Given a signature $\Sigma$, a \emph{$\Sigma$-Nmatrix} \index{non-deterministic matrix (Nmatrix)}  is a tuple $\Mt=\tuple{A,\cdot_\Mt,D}$ where $A$ is the set of \emph{truth values}\index{truth values}, or \emph{elements},
 $\bA_\Mt=\tuple{A,\cdot_\Mt}$ is a $\Sigma$-multialgebra, and
 $D\subseteq A$ is a set of \emph{designated elements}\index{designated elements}. When 
 $\bA_\Mt$ is an ordinary algebra, $\Mt$ is just a (deterministic) logical matrix.

A \emph{valuation} \index{valuation}  on $\Mt$
is an element of $\Val(\Mt)=\mathsf{hom}(\Fm,\bA_\Mt)$,
that is, a function $v:Fm\to A$ such that
$v(\conn(\varphi_1,\ldots,\varphi_k))\in
\conn_\Mt(v(\varphi_1),\ldots,v(\varphi_k))$
for every $\conn\in \Sigma^{k}$ and $\varphi_1,\ldots,\varphi_k\in Fm  
$.

\end{definition}

In contrast to the deterministic case, a valuation $v$ on a Nmatrix $\Mt=\tuple{A,\cdot_\Mt,D}$ is not completely determined by the values 
assigned to propositional variables.  
However, it is  true that if $\Gamma=\sub(\Gamma)$ is a set of formulas closed under subformulas then a 
function 
$w:\Gamma\to A$ satisfying $w(\conn(\varphi_1,\ldots,\varphi_k))\in\conn_\Mt(w(\varphi_1),\ldots,w(\varphi_k))$ for every complex formula $\conn(\varphi_1,\ldots,\varphi_k)\in\Gamma$, or \emph{prevaluation}\index{prevaluation},
can always be extended to a valuation~\cite{avlev05}.

\smallskip

Any Nmatrix $\Mt$  
induces the multiple-conclusion logic
$\der_\Mt {\subseteq}\, \wp(Fm)\times \wp(Fm)$ defined
by ${\Gamma}{\der_\Mt} {\Delta}$ if and only if 
$v[\Gamma]\subseteq D$ implies $v[\Delta]\cap D\neq \emptyset$ for every
$v\in \Val(\Mt)  
$.
Since logics over a given signature 
are closed for arbitrary intersections,  
we extend this notion to any class $\cM=\{\Mt_i:i\in I\}$ of $\Sigma$-Nmatrices in the expected way,
by setting $\der_\cM=\bigcap_{i\in I}\der_{\Mt_i}$.
However, as observed in~\cite{newfibring}, almost every logic can be given by a single Nmatrix, a sufficient condition being that $\Sigma$ contain at least a connective of arity greater than $1$.
The  
single-conclusion relation associated to 
$\cM$ can be recovered by $\vdash_\cM{=}\vdash_{\der_\cM}$.

Given a multiple-conclusion logic $\der$,
we say that an Nmatrix  $\Mt$ is \emph{$\der$-sound} \index{$\der$-sound}  whenever $\der{\subseteq} \der_\Mt$.
We further denote by $\Nmatr(\der)$ the class of all $\der$-sound Nmatrices 
and by  $\Matr(\der)$ the class
of
 deterministic Nmatrices in $\Nmatr(\der)$.

\subsection{Illustrations: Logiken ohne Eigenschaften}\label{sec:noprops}

The paper~\cite{font2013} by J.M.~Font considers a most simple (single-conclusion) logic, dubbed $\mathcal{I}$, whose language consists of a single binary connective $\to$, axiomatized by the following well-known rule schemata 
(the identity axiom and \emph{modus ponens}):
$$
  \frac{}{p\to p} \ \mathsf{r}_{\mathsf{id}} \qquad \qquad  \qquad 
  \frac{p\,,\,p\to q}{q} \ \mathsf{r}_{\mathsf{mp}}
$$

From the point of view of~\cite{font2013} -- which is that of Abstract Algebraic Logic --
the main interest in  this system stems from the observation that  $\mathcal{I}$ may be described as a \emph{Logik ohne Eigenschaften}: in Font's words, \guillemotleft the properties of this logic are worth studying, even if the conclusion of such study may be that ``this logic has almost no properties''\guillemotright~\cite[p.~436]{font2013}.
In fact, rather than having almost no properties,   
one might more accurately say that Font's logic enjoys a number of
interesting ``negative properties'', i.e.~$\mathcal{I}$ is a particularly simple witness to the failure of certain properties
that are very widespread among the logics known in the algebraic community as \emph{protoalgebraic} (more on this below).

In this subsection we are going to  take a look  both at Font's logic and at  logics over the same language  
obtained by removing one or both of the above rule schemata:  we shall thus be dealing with very weak systems,
which lack certain properties that even
$\mathcal{I}$ has (e.g.~protoalgebraicity). 
These  
will serve  as examples for illustrating  the new possibilities brought by non-deterministic matrices. 
We shall first consider the logic given by the empty set of rules
($R_\mathsf{u}=\{\}$)
and  the one defined by  \emph{modus ponens}\index{modus ponens}  alone 
($R_\mathsf{mp}= 
\{  \mathsf{r}_{\mathsf{mp}} \}
$).

\begin{example}\label{ex2val}
Consider the multiple-conclusion logics $\der_{R_\mathsf{u}}$ and $\der_{R_\mathsf{mp}}$. 
We note that the corresponding single-conclusion companions are not locally tabular, hence, not given by any finite set of finite matrices~\cite{finval}.\smallskip

For $n,m\in \nats\cup\{\omega\}$ let $U_n=\{\bot_i:0\leq i < n\}$, $D_m=\{\top_i:0\leq i < m\}$ and
$A_{n,m}=U_n\cup D_m$.
Note that the cardinality of $A_{n,m}$ is $n+m$ whenever $n,m\in \nats$, otherwise it is denumerable.

 Let us describe some Nmatrices in $\Nmatr(\der_{R})$ for $R\in \{R_\mathsf{u},R_\mathsf{mp}\}$. Consider the following families of Nmatrices, where $x,y\in A_{n,m}$  { 
(To improve readability, we abbreviate $\cdot_{\mathsf{u}}=\cdot_{\mathbb{U}_{n,m}}$
 and $\cdot_{\mathsf{mp}}=\cdot_{\mathbb{MP}_{n,m}}$. No confusion may arise, for one can consider  $n$ and $m$  fixed; similar conventions are used throughout the paper).} \\

$\mathbb{U}_{n,m}=\tuple{A_{n,m},\cdot_\mathsf{u},D_m}$ with
$\to_\mathsf{u}(x,y)=A_{n,m}$;\\

$\mathbb{MP}_{n,m}=\tuple{A_{n,m},\cdot_\mathsf{mp},D_m}$ with $\to_\mathsf{mp}(x,y)=
\begin{cases}
    U_n & \text{ if }x\in D_m \text{ and }y\in U_n\\
    A_{n,m} & \text{ otherwise.}
\end{cases}$\\

Note that $\der_{R_{\mathsf{u}}}{\subseteq} \der_{\mathbb{U}_{n,m}}$ and
$\der_{R_{\mathsf{mp}}}{\subseteq} \der_{\mathbb{MP}_{n,m}}$.

Indeed, in~\cite{wollic,wollic19,newfibring} we have shown that $\der_{R_{\mathsf{u}}}{=}\der_{\mathbb{U}_{1,1}}$
and $\der_{R_{\mathsf{mp}}}{=}\der_{\mathbb{MP}_{1,1}}$.

As $\mathbb{U}_{1,1}$ is a  
 { 
 subNmatrix (see Definition~\ref{def:6})} of all the Nmatrices in the $\mathbb{U}_{n,m}$ family, and the same is true about $\mathbb{MP}_{1,1}$ in the $\mathbb{MP}_{n,m}$ family, it is straightforward to conclude that
$\der_{R_{\mathsf{u}}}{=}\der_{\mathbb{U}_{1,1}}{=}\der_{\mathbb{U}_{n,m}}$, and
  $\der_{R_{\mathsf{mp}}}{=}\der_{\mathbb{MP}_{n,m}}{=}\der_{\mathbb{MP}_{1,1}}$ for $n,m>0$.

 ~ \hfill$\triangle$

\end{example}

Of course, finite Nmatrices cannot be expected to characterize all possible logics, as we next illustrate.

\begin{example}\label{noprops2}

 Let 
 $R_\mathsf{id}=\{r_\mathsf{id}\}$  
and consider 
$\mathbb{D}_{n,m}=\tuple{A_{n,m},\cdot_\mathsf{id},D_m}$  with  { 
(
$\cdot_{\mathsf{id}}=\cdot_{\mathbb{D}_{n,m}}$ and) }
$$\to_\mathsf{id}(x,y)=
\begin{cases}
    D_m & \text{ if }x=y\\
    A_{n,m} & \text{ otherwise. }
\end{cases}$$

First we note that %
$\der_{R_{\mathsf{id}}}{\subseteq}\der_{\mathbb{D}_{n,m}}$ for $n,m\in \nats\cup\{\omega\}$. Let us further check that (i) $\vdash_{R_{\mathsf{id}}}{=}\vdash_{\mathbb{D}_{1,2}}$, (ii)
   $\der_{R_{\mathsf{id}}}{=}\der_{\mathbb{D}_{\omega,\omega}}$, and
(iii) $\der_{R_{\mathsf{id}}}{\subsetneq}\der_{\mathbb{D}_{n,m}}$ if $n\in \nats$ or $m\in\nats$.\\

(i) To see that $\vdash_{R_{\mathsf{id}}}{=}\vdash_{\mathbb{D}_{1,2}}$,   assume that $\Gamma\not\vdash_{R_{\mathsf{id}}} \varphi$. We know that $\varphi\notin \Gamma$, and either $\varphi\in P$ or 
$\varphi=\varphi_1\to\varphi_2$ with $\varphi_1\neq \varphi_2$. 
We will show that in any case there is $v\in \Val(\mathbb{D}_{1,2})$
with $v[\Gamma]\subseteq D_2$ and $v(\varphi)\in U_1$.
If $\varphi\in P$ consider
$$v(\psi)=
\begin{cases}
 \bot_0&\mbox{ if }\psi=\varphi\\
  \top_0 &\mbox{ if }\psi\neq\varphi.
\end{cases}
$$
It is clear that $v(\psi_1\to\psi_2)=\top_0\in {D_2}\, {\subseteq} \to_{\mathsf{id}}(v(\psi_1),v(\psi_2))$ for all $\psi_1,\psi_2\in Fm$, even if $v(\psi_1)\neq v(\psi_2)$, and thus $v\in\Val(\mathbb{D}_{1,2})$. 

Further, $v[\Gamma]=\{\top_0\}\subseteq D_2$, $v(\varphi)=\bot_0\in U_1$, and we conclude that $\Gamma\not\vdash_{\mathbb{D}_{1,2}}\varphi$.
\smallskip

If instead $\varphi=\varphi_1\to\varphi_2$ with 
$\varphi_1\neq \varphi_2$, 
we can consider
  $$v(\psi)=
\begin{cases}
 \bot_0 &\mbox{ if }\psi=\varphi\\
\top_0 &\mbox{ if }\psi=\varphi_1\\
\top_1 &\mbox{ otherwise. } 
\end{cases}$$

Easily, $v(\varphi)=v(\varphi_1\to\varphi_2)=\bot_0\,{\in}\to_{\mathsf{id}}(v(\varphi_1),v(\varphi_2))$ because $v(\varphi_1)\neq v(\varphi_2)$, as necessarily $v(\varphi_1)=\top_0$ and $v(\varphi_2)=\top_1$.

Further, if either $\psi_1\neq\varphi_1$ or $\psi_2\neq\varphi_2$ then 
$v(\psi_1\to\psi_2)\in\{\top_0,\top_1\}\subseteq {D_2}\, {\subseteq} \to_{\mathsf{id}}(v(\psi_1),v(\psi_2))$, even if $v(\psi_1)\neq v(\psi_2)$, and thus $v\in\Val(\mathbb{D}_{1,2})$. 

Moreover, 
$v[\Gamma]\subseteq\{\top_0,\top_1\}\subseteq D_2$, $v(\varphi)=\bot_0\in U_1$, and so $\Gamma\not\vdash_{\mathbb{D}_{1,2}}\varphi$.
\smallskip

(ii) Assuming $\Gamma\not\der_{R_{\mathsf{id}}}\Delta$ and using cut for sets, we know that there exists $\Omega\subseteq Fm$ with $\Gamma\subseteq\Omega$, $\Delta\subseteq \overline{\Omega}=Fm\setminus\Omega$ such that
$\Omega\not\der_{R_{\mathsf{id}}}\overline{\Omega}$. We fix some enumeration of formulas $e:Fm\to \nats$  and consider 
  $$v(\varphi)=\begin{cases}
      \top_{e(\varphi)} &\text{ if }\varphi\in \Omega\\
            \bot_{e(\varphi)} &\text{ otherwise.}
  \end{cases}$$ Since $\psi\to\psi\in \Omega$ for every $\psi\in Fm$ we conclude that $v(\psi\to\psi)\in\, \to_{\mathsf{id}}(v(\psi),v(\psi))$.
  Since $v\in\Val(\mathbb{D}_{\omega,\omega})$, $v[\Gamma]\subseteq D_\omega$ and $v[\Delta]\subseteq U_\omega$ we confirm that
$\Gamma\not\der_{\mathbb{D}_{\omega,\omega}}\Delta$.\smallskip

(iii) For $k\in \nats$, let $\Gamma_k=\{p_i:0\leq i \leq k\} $ and $\Delta_k=\{p_i\to p_j:0\leq i< j\leq k\}$.\smallskip

It is easy to see that 
$\Gamma_m \der_{\mathbb{D}_{n,m}} \Delta_m$ if $m\in \nats$, as for any valuation such that $v[\Gamma_m]\subseteq D_m$, by the pigeonhole principle, there must exist $0\leq i< j\leq m$ with $v(p_i)=v(p_j)$ and thus with $v(p_i\to p_j)\in D_m$.

Analogously, 
$ \der_{\mathbb{D}_{n,m}} \Gamma_n\cup\Delta_n$ if $n\in \nats$, as for any valuation such that $v[\Gamma_n]\subseteq U_n$, by the pigeonhole principle, there must exist $0\leq i< j\leq n$ with $v(p_i)=v(p_j)$ and thus with $v(p_i\to p_j)\in D_m$.

Of course, we have that $\Gamma_k \not\der_{\mathbb{D}_{\omega,\omega}} \Delta_k$ and $ \not\der_{\mathbb{D}_{\omega,\omega}} \Gamma_k\cup\Delta_k$ for every $k\in\nats$, and thus  $\der_{R_{\mathsf{id}}}{\subsetneq}\der_{\mathbb{D}_{n,m}}$ if $n\in \nats$ or $m\in\nats$.\\

~\hfill$\triangle$
\end{example}

\smallskip

\begin{example}\label{ex:proto}

 {

Let now
$R_\mathcal{I}=\{r_\mathsf{id},r_\mathsf{mp}\}$.  
From the standpoint of Abstract Algebraic Logic, the single-conclusion
logic $\vdash_{R_\mathcal{I}}$ 
 is interesting in that it is a non-trivial but very weak and most simple  member of the family of \emph{protoalgebraic logics}\index{protoalgebraic logics}. The latter,
extensively studied by Czelakowski (see e.g.~\cite{czela2001}),
are the basis of the so-called \emph{Leibniz hierarchy}, and
 arguably form
 the widest
class of logics
to which general algebraic methods may be successfully applied.

Protoalgebraicity may be characterized in several alternative ways, among which the following one provides 
a justification for the special interest in the logic $\mathcal{I}$: a single-conclusion logic $\vdash$ is protoalgebraic 
if it possesses a set $\Delta(x,y)$ of formulas in at most two variables satisfying 
$\vdash \Delta(x,x)$ and
$x, \Delta(x,y) \vdash y$.
It is natural to think of $\Delta(x,y)$ as a generalized implication (or biconditional) connective;
 the two previous schemata may thus be read as generalized forms of, respectively, the identity axiom and
the \emph{modus ponens} rule.

As noted by Font~\cite[p.~Prop.~5.2]{font2013}, 
the logic 
$\mathcal{I}$ does not have an algebraic semantics -- let alone an equivalent one in Blok and Pigozzi's sense --  
and it is not finite-valued in any sense, i.e.~it is not determined by any single finite matrix (deterministic or not; indeed, as we shall see,
not even a finite set of finite Nmatrices would suffice).} Consider the following family of Nmatrices, with $n,m\in \nats\cup\{\omega\}$,  { 
($\cdot_{\mathsf{I}}=\cdot_{\mathbb{I}_{n,m}}$ and)}:\\

$\mathbb{I}_{n,m}=\tuple{A_{n,m},\cdot_\mathsf{I},D_m}$ with 
$\to_\mathsf{I}(x,y)=
\begin{cases}
    D_m & \text{ if }x=y \\
    U_n & \text{ if }x\in D_m \text{ and }y\in U_n\\
    A_{n,m} & \text{ otherwise. }
\end{cases}$\\

Note that 
$\der_{R_{\mathsf{I}}}{\subseteq}\der_{\mathbb{I}_{n,m}}$, in all cases.
 
 {
 The proof that $\der_{R_{\mathsf{I}}}{=}\der_{\mathbb{I}_{\omega,\omega}}$  
 is analogous to that in Example ~\ref{noprops2} (ii) and shall be omitted.} 
Let us show that
$\der_{R_{\mathsf{I}}}{\subsetneq}\der_{\mathbb{D}_{n,m}}$ if $n\in \nats$ or $m\in\nats$. For $k\in \nats$, consider:
\begin{align*}
 \Gamma_k&=\{p_i:0\leq i\leq k \}\cup \{(p_i\to p_j)\to p_{k+1}:0\leq i<j\leq k\},\textrm{ and}\\
\Delta_k&=\{p_i\to p_{k+1}:0\leq i\leq k\}\cup \{(p_i\to p_j)\to p_{k+1}:0\leq i<j\leq k\}.
\end{align*}

It is simple to see that 
$\Gamma_m \der_{\mathbb{I}_{n,m}} p_{m+1}$ if $m\in \nats$, as for any valuation such that $v[\Gamma_m]\subseteq D_m$, by the pigeonhole principle, there must exist $0\leq i< j\leq m$ with $v(p_i)=v(p_j)$ and thus with $v(p_i\to p_j)\in D_m$, and since $(p_i\to p_j)\to p_{m+1}\in\Gamma_m$ it must be the case that $v(p_{m+1})\notin U_n$, i.e., $v(p_{m+1})\in D_m$.

In a similar way,  
$\Delta_n \der_{\mathbb{I}_{n,m}} p_{n+1}$ if $n\in \nats$, as for any valuation such that $v(\Delta_n)\subseteq D_m$, 
either $v(p_{n+1})\in D_m$ or else $v(p_i)\in U_n$ for all $0\leq i\leq n$. But then, by the pigeonhole principle, there must exist $0\leq i< j\leq m$ with $v(p_i)=v(p_j)$ and thus with $v(p_i\to p_j)\in D_m$, and since $(p_i\to p_j)\to p_{n+1}\in\Delta_n$ it must be the case that $v(p_{n+1})\notin U_n$, i.e., $v(p_{n+1})\in D_m$ anyway.

Of course, we have that $\Gamma_k \not\der_{\mathbb{I}_{\omega,\omega}} p_{k+1}$ and $ \Delta_k\not\der_{\mathbb{I}_{\omega,\omega}} p_{k+1}$ for every $k\in\nats$, and thus  $\der_{R_{\mathsf{I}}}{\subsetneq}\der_{\mathbb{I}_{n,m}}$ if $n\in \nats$ or $m\in\nats$.
  \hfill$\triangle$
\end{example}

\begin{remark}\label{rem:noprops}

It is known that, for a single-conclusion logic,
being characterizable by a single finite matrix boils down to the properties of cancellation, local tabularity and 
finite determinedness. Cancellation and local tabularity are well known 
{
(see, respectively,~\cite{SS1971}
and 
\cite[p.~6]{DBLP:books/daglib/0030819})}. The notion of finite determinedness, 
 introduced in~\cite{finval}, can be formulated as follows~\cite[Def.~3.1, Lemma~3.2]{finval}.
A logic is \emph{finitely determined}\index{finitely-determined} if and only if 
there exists a  natural number $n$ such that, if $\Gamma^\sigma \der \varphi^\sigma $
for every $\sigma:P\to \{ p_1,\ldots,p_n \}$,  then 
$\Gamma \der \varphi $.

For  Nmatrices we do not have any similar characterization, and indeed any of the above-mentioned properties 
might fail.
\smallskip

For any $\Sigma$-Nmatrix $\Mt$
such that $\Sigma$ does not contain any $0$-ary connectives, one can easily check
that $\der_\Mt$ satisfies cancellation, and therefore there must exist a $\Sigma$-matrix $\Mt'$ such that
$\der_{\Mt}{=}\der_{\Mt'}$, even if $\Mt'$
might be non-denumerable~\cite{ShoesmithSmiley}.
However, if in $\Mt$ the interpretation of a $0$-ary connective $\botop$ contains both designated and non-designated elements, then
cancellation fails, for  $\botop\der_\Mt\botop$ and $\not\der_\Mt\botop$ but $\var(\botop)\cap\var(\botop)=\emptyset$ (see~\cite{newfibring,MaCaRi2018PlugandPlay}).\smallskip

We can use the Nmatrices in the previous examples to illustrate how logics induced by a finite Nmatrix may also fail to satisfy the other above-mentioned properties. Consider the set of variables $P_k=\{p_i:0\leq i< k\}$ for $k\in \nats$.

Regarding local tabularity, note for instance that
$\dashv\vdash_{\mathbb{U}_{n,m}}$  for $n,m\geq 1$ divides $Fm_k=\{\varphi\in Fm:\var(\varphi)\subseteq P_k\}$
in an infinite number of classes. Namely, no two distinct formulas $\varphi,\psi\in Fm_k$ are interderivable in the logic induced by $\mathbb{U}_{n,m}$.

Regarding finite determinedness,
 the logic $\der_{R_\mathsf{mp}}$ 
is characterizable by the 2-valued Nmatrix $\mathbb{MP}_{1,1}$, but it 
is not $k$-determined for any $k\in\nats$. 
Namely, for each $k$, consider
$$\Gamma_k=\{p_i\to p_j: 0\leq i<j\leq k\}\cup\{(p_i\to p_i)\to p_{k+1}: 0\leq i\leq k\}.$$
We have that $\Gamma_k\not\der_{\mathbb{MP}_{1,1}} p_{k+1}$.
However, by the pigeonhole principle, for every $\sigma:P\to P_k$
we must have $0\leq i<j\leq k$ such that $\sigma(p_i)=\sigma(p_j)=p_\ell\in P_k$.
Hence $(p_i\to p_j)^\sigma=p_\ell\to p_\ell\in \Gamma_k^\sigma$, and also
$((p_i\to p_i)\to p_{k+1})^\sigma=(p_\ell\to p_\ell)\to \sigma(p_{k+1})\in \Gamma_k^\sigma$,
and so $\Gamma_k^\sigma\der_{R_\mathsf{mp}} \sigma(p_{k+1})$.
\end{remark}

Next we present a family of interesting examples of natural denumerable Nmatrices for very weak logics.

\begin{example}
Given a set of axioms $\Ax\subseteq Fm$ over an arbitrary signature $\Sigma$, let $R_\Ax=\{\frac{}{\varphi}:\varphi\in \Ax\}$ and  $\Ax^{\textsf{inst}}=\{\varphi^\sigma:\varphi\in \Ax,\sigma:P\to Fm\}$.
Consider the $\Sigma$-Nmatrix $\Mt_{\Ax}=\tuple{A_\Ax,\cdot_\Ax,D_\Ax}$ where 
{
($\cdot_{\Ax}=\cdot_{\Mt_{\Ax}}$ and)}:
\begin{align*}
 A_\Ax&=((Fm\setminus \Ax^{\textsf{inst}})\times\{0\})\cup (Fm\times\{1\}),\\
 D_\Ax&=Fm\times\{1\},\\
 \conn_\Ax((\varphi_1,i_1),\ldots,(\varphi_k,i_k))&=\{(\conn(\varphi_1,\ldots,\varphi_k),i):i=0,1\}\cap A_\Ax.
\end{align*}
for each $k<\omega$, $\conn\in\Sigma_k$, $i_1,\dots,i_k\in\{0,1\}$, and $\varphi_1,\dots,\varphi_k\in Fm$.

It was shown in~\cite{OnAxRexp,newfibring} that $\der_{R_\Ax}{=}\der_{\Mt_{\Ax}}$, and thus $\vdash_{R_\Ax}{=}\vdash_{\Mt_{\Ax}}$. 
  \hfill$\triangle$
\end{example}
 
By allowing both non-determinism and partiality, one obtains
the even more general notion of a \emph{partial non-deterministic matrix (PNmatrix)}~\cite{Baaz2013}\index{partial non-deterministic matrix (PNmatrix)}:
in this case the underlying multialgebra $\bA$ of a given matrix may be partial in that
$\conn_\bA:A^k\to \wp(A)$
 can yield an empty set for some $\conn\in \Sigma^{k}$.
Partiality is beyond the scope of this paper, but it is worth noting that in this more general setting  we can obtain semantics based on a single denumerable PNmatrix for a 
wider range of logics, including intuitionistic logic and every modal logic~\cite{OnAxRexp}.

\section{Constructions on Nmatrices}\label{sec:nmatrices}

In this section we show that many of the usual constructions on standard logical matrices can be extended to 
Nmatrices, and establish some basic properties about them; we will  highlight  the main differences between
the standard setting and ours in the appropriate places.

\subsection{Strict homomorphisms, preimages and quotients}
 
We start by introducing a  notion of homomorphism of Nmatrices that extends the usual one for logical matrices. We then explore its properties, and discuss the similarities and differences with respect to the deterministic setting.

\begin{definition}\label{def:6}
Let $\Mt_1= \tuple{A_1,\cdot_1,D_1}$ and
$\Mt_2= \tuple{A_2,\cdot_2,D_2}$ be $\Sigma$-Nmatrices.
A \emph{homomorphism}\index{Nmatrix homomorphism} $h:\Mt_1\to\Mt_2$ is a function $h:A_1\to A_2$ 
such that $h$ is a multialgebra homomorphism between $\tuple{A_1,\cdot_1}$ and $\tuple{A_2,\cdot_2}$ and $h(D_1)\subseteq D_2$.

If $h^{-1}(D_2)=D_1$ we say $h$ is \emph{strict}\index{strict homomorphism}, and if it is also injective we say that it is an \emph{embedding}.
When $h$ is a surjective function we say $h$ is \emph{onto}.
\end{definition}

A strict homomorphism $h:\Mt_1\to\Mt_2$ is an \emph{isomorphism} whenever it is a bijective function and
$h(\conn_{\Mt_1}(x_1,\ldots,x_k))=\conn_{\Mt_2}(h(x_1),\ldots,h(x_k))$ for every $k<\omega$, $x_1,\ldots,x_k\in A_1^k$
and $\conn\in \Sigma^{k}$.

The \emph{image} of an Nmatrix  $\Mt_1$ under a
strict homomorphism $h:\Mt_1\to \Mt_2$ is the Nmatrix $h[\Mt_1]=\tuple{h[A_1],\cdot_{h},D_2\cap h[A_1]}$
where $h[A_1]=\{h(a):a\in A_1\}$ 
{
(as before, we simplify $\cdot_{h}=\cdot_{{h[\Mt_1]}}$)} and
 \begin{align*}
 \conn_{h}(x_1,\ldots,x_k)&=\{h(y): y\in \conn_1(y_1,\ldots,y_k),h(y_i)=x_i,1\leq i\leq k\}.
 \end{align*}
{
As usual, we say that  $\Mt_1$ is the \emph{preimage}\index{preimage}  of $h[A_1]$ under $h:\Mt_1\to\Mt_2$. Preimages under strict homomorphisms are also called \emph{rexpansions}\index{rexpansions}  in~\cite{rexpansions}. 
 }

Given $\Sigma$-Nmatrices $\Mt_1= \tuple{A_1,\cdot_1,D_1}$ and
$\Mt_2= \tuple{A_2,\cdot_2,D_2}$,
we say that $\Mt_1$ is a \emph{subNmatrix}\index{subNmatrix} of  $\Mt_2$
whenever $A_1\subseteq A_2$ and for every $\conn\in \Sigma^{k}$ and
$x_1,\dots,x_k\in A_1$, $\conn_1(x_1,\ldots,x_k)\subseteq\conn_2(x_1,\ldots,x_k)$.
That is, the function $h:A_1\to A_2$ defined as $h(x)=x$ induces an embedding $h:\Mt_1\to\Mt_2$.
Furthermore, whenever for every $\conn\in \Sigma^{k}$ and
$x_1,\dots,x_k\in A_1$, $\conn_1(x_1,\ldots,x_k)=\conn_2(x_1,\ldots,x_k)$,
  we say that $\Mt_1$ is the \emph{restriction}\index{restriction} of $\Mt_2$ to $A_1$.

For every strict homomorphism $h:\Mt_1\to\Mt_2$,
 $h[\Mt_1]$ is a subNmatrix of $\Mt_2$, and when
  $h[\Mt_1]=\Mt_2$ we say that $h$ is a \emph{covering} homomorphism \index{covering homomorphism} (which implies that $h$ is also onto).

 \begin{definition}
 Given a class $\cM$ of Nmatrices,  we denote by 
$\SubN(\cM)$ 
the class of isomorphic copies of subNmatrices  of Nmatrices in $\cM$,
and by 
 $\Rexp(\cM)$ 
the class of isomorphic copies of preimages by strict homomorphisms of Nmatrices in $\cM$. 
\end{definition}
{
It follows from the definitions that there is a strict homomorphism from $\Mt_1$ to $\Mt_2$ if and only if $\Mt_1\in \Rexp(\SubN(\Mt_2))$, and there is a covering strict homomorphism if and only if $\Mt_1\in \Rexp(\Mt_2)$.}

\paragraph*{Lindenbaum matrices.} Lindenbaum matrices can  be seen as preimages of valuations also in the realm of Nmatrices.

\begin{definition}\label{def:LindMatr}
The set of \emph{Lindenbaum matrices} of a multiple-conclusion logic \index{Lindenbaum matrices of a multiple-conclusion logic}  $\der$ is 
 $$\cL_\der =\{\tuple{Fm,\cdot_\Fm,\Gamma}:\Gamma\not\der Fm\setminus \Gamma\}.$$ 

The Lindenbaum matrix \index{Lindenbaum matrix induced by a valuation} induced by a valuation $v$ over $\Mt$, 
is $\mathbb{L}_v=\tuple{Fm,\cdot_\Fm,D_v}$ with $D_v=v^{-1}(D)$.

\end{definition}

%

The next lemma reflects the fact that a valuation on an Nmatrix 
induces a strict homomorphism whose image sits inside a  denumerable subNmatrix.

\begin{lemma}\label{valrexp}
Let $\Mt$ be a $\Sigma$-Nmatrix, $\cM$ a class of $\Sigma$-Nmatrices, and $\der$ a logic.

If $\Mt\in\Nmatr(\der)$ and $v\in\Val(\Mt)$ then $\mathbb{L}_v\in \cL_\der$.

Further, $\cL_{\der_\cM}=\{\mathbb{L}_v:v\in \Val(\Mt),\Mt\in \cM\}\subseteq {\Rexp\SubN(\cM)}$. 
\end{lemma}
\begin{proof}
It is straightforward that if $\Mt\in\Nmatr(\der)$ and $v\in\Val(\Mt)$ then we have
$D_v\not\der(Fm\setminus D_v)$, and thus $\mathbb{L}_v\in \cL_\der$.

Additionally, if $\Gamma\not\der_{\cM} (Fm\setminus \Gamma)$ then there is $\Mt=\tuple{A,\cdot_{\Mt} ,D}\in \cM$ and
$v\in \Val(\Mt)$ such that
$v^{-1}(D)=\Gamma$. 
 Obviously, 
 $v:\mathbb{L}_v\to\Mt$ is a strict homomorphism {and $h[\mathbb{L}_v]\in \SubN(\Mt)$}.
\end{proof}

\paragraph*{Images of onto strict homomorphisms can induce a weaker logic.} It is well known that, if $\Mt_1,\Mt_2$ are matrices and $h:\Mt_1\to \Mt_2$ is a strict homomorphism, then $\vdash_{\Mt_2}{\subseteq}\vdash_{\Mt_1}$ and, if $h$ is also onto, then $\vdash_{\Mt_2}{\subseteq}\vdash_{\Mt_1}$ (see, for instance,~\cite{Fo16}). 
These properties still hold, for similar reasons, if one considers multiple-conclusion logics. However, while strict homomorphisms of Nmatrices do guarantee that the logic of $\Mt_2$ is weaker than that of $\Mt_1$, surjectivity does not force equality anymore.

\begin{lemma}\label{lem:rexpstronger}
If $h:\Mt_1\to \Mt_2$ is a strict homomorphism between $\Sigma$-Nmatrices $\Mt_1$ and $\Mt_2$, then $\der_{\Mt_2}{\subseteq}\der_{\Mt_1}$.
{ 
In other words, if $\Mt_1\in \Rexp(\SubN(\Mt_2))$ then $\der_{\Mt_2}{\subseteq}\der_{\Mt_1}$.
}
\end{lemma}
\begin{proof}
The result follows easily by noting that if $v\in\Val(\Mt_2)$
 then $h\circ v \in\Val(\Mt_1)$ designates exactly the same formulas as $v$.  
\end{proof}

\begin{example}\label{ex:hom}
Recall the examples in Section~\ref{sec:noprops}, and consider $h:A_{\omega,\omega}\to A_{1,1}$
given by $$h(x)=\begin{cases}
    \bot_0&\text{ if }x\in U_n\\
    \top_0&\text{ if }x\in D_m.
\end{cases}$$

We know that $\der_{R_\mathsf{u}}{=}\der_{\mathbb{U}_{1,1}}{\subsetneq} \der_{\mathbb{D}_{\omega,\omega}}{=}\der_{R_\mathsf{id}}$. Still, note that $h$ constitutes an onto (and covering) strict homomorphism $h:\mathbb{D}_{\omega,\omega}\to\mathbb{U}_{1,1}$.
  \hfill$\triangle$
\end{example}

Knowing that this basic fact about homomorphisms of matrices fails for Nmatrices, it is  worth mentioning that no suitably adapted stronger notion seems to play the same role.
Note, for instance, that the homomorphism given in the example is covering.
Since images of strict homomorphisms (or quotients) play a fundamental role in matrix semantics and  the theory of algebraic logic, we will need to take extra care in dealing with Nmatrices.

\paragraph*{Quotients of Nmatrices.} Strict homomorphisms between matrices  are associated
 with 
quotients by congruences compatible with the set of designated truth values. In the setting of Nmatrices there appears to be no reasonable notion of congruence playing the same role. Indeed, quotients of Nmatrices may be obtained through any 
equivalence relation on its set of truth values.

\begin{definition}\label{quo1}
 Let $\Mt=\tuple{A,\cdot_\Mt,D}$ be a $\Sigma$-Nmatrix, and $\equiv$ an equivalence relation on $A$. The \emph{quotient of $\Mt$ by $\equiv$}\index{quotient of Nmatrices} is the Nmatrix ${\Mt_{/\!\equiv}}=\tuple{A_{/\!\equiv},\cdot_\equiv,D_{/\!\equiv}}$ where, for each $k<w$ and $\conn\in\Sigma^k$:
\begin{align*}
    \conn_\equiv([x_1]_\equiv,\ldots,[x_k]_\equiv)&=
\{[y]_\equiv:y\in \conn_\Mt(y_1,\ldots,y_k),y_i\in [x_i]_\equiv,1\leq i\leq k\}.
\end{align*}
\end{definition}

Quotients are of course more interesting if the equivalence relation $\equiv$ is \emph{compatible}\index{compatible equivalence relation} (with $D$), that is, $x\equiv y$ implies that
 $x\in D$ if and only if $y\in D$. When this is the case, we say that ${\Mt_{/\!\equiv}}$ is a 
 \emph{compatible quotient}. 
 
 As with (deterministic) matrices, we have that compatible quotients of Nmatrices are isomorphic to images by
  strict homomorphisms. 

Given a strict homomorphism $h:\Mt_1\to \Mt_2$ between Nmatrices $\Mt_1=\tuple{A_1,\cdot_1,D_1}$ and 
$\Mt_2=\tuple{A_2,\cdot_2,D_2}$, let
 $\equiv_h{\subseteq} A_1\times A_1$ be the kernel equivalence relation 
defined as $x\equiv_h y$ iff $h(x)=h(y)$. Clearly, $\equiv_h$ is compatible (with $D_1$).
The following result is a straightforward consequence of the definitions.

\begin{lemma}\label{quo2}
Let $\Mt=\tuple{A,\cdot_\Mt,D}$ be a $\Sigma$-Nmatrix,  
$\equiv$ an equivalence relation on $A$.

 The natural function $h:A\to A_{/\equiv}$ given by $h(a)=[a]_\equiv$ constitutes a covering homomorphism $h:\Mt\to\Mt_{/\equiv}$. Additionally, $h:\Mt\to\Mt_{/\equiv}$ is strict if and only if $\equiv$ is compatible.

Further,
if $h:\Mt_1\to \Mt_2$ is a strict homomorphism of $\Sigma$-Nmatrices then 
$h[\Mt_1]$ is isomorphic to $\Mt_{1/\!\equiv_h}$.
\end{lemma}

Despite the  above lemma, it may  happen that the image of an onto strict homomorphism does not cover the target Nmatrix; and, even when it does, it  may define a weaker logic than that of the original Nmatrix.
We illustrate this in the next Example.

\begin{example}\label{smashintobool} 
 Every 
 Nmatrix can be factored into an Nmatrix with at most two values
 by collapsing together all the designated elements on the one hand, and  all the non-designated ones on the other. However, such quotients may not be covering.

Recall the Nmatrices $\mathbb{U}_{1,1}$ and $\mathbb{D}_{n,m}$ from Examples~\ref{ex2val} and~\ref{noprops2}.
Consider the functions
$h_{n,m}:A_{n,m}\to A_{1,1}$, for
 $n,m\in \nats$, defined by 
$$h_{n,m}(x)=\begin{cases}
  \bot_0  & \text{ if }x\in U_n \\
  \top_0  &\text{ if }x\in D_m.
\end{cases}$$

It is obvious that $h_{n,m}:\mathbb{D}_{n,m}\to\mathbb{U}_{1,1}$ is always a strict homomorphism. Further, if $n,m>0$, then these homomorphisms are onto. 
However, they may fail to be covering. Take, for instance, $h_{1,2}$ and $h_{2,1}$. The corresponding images $h_{1,2}[\mathbb{D}_{1,2}]$ and $h_{2,1}[\mathbb{D}_{2,1}]$ 
correspond (up to isomorphism) to two different subNmatrices of $\mathbb{U}_{1,1}$ whose tables can be given, respectively, by
\begin{center}
    \begin{tabular}{c|c c }
    $\to$ & $0$ & $1$ \\
    \hline
    $0$  & $0,1$ & $0,1$ \\
     $1$   & $0,1$ &$1$
\end{tabular}
\qquad
\begin{tabular}{c|c c }
    $\to$ & $0$ & $1$ \\
    \hline
    $0$  & $1$ & $0,1$ \\
     $1$   & $0,1$ &$0,1$
\end{tabular}
\end{center}
where, on the left $1=[\top_0]_{\equiv_{h_{1,2}}}=[\top_1]_{\equiv_{h_{1,2}}}$ and $0=[\bot_0]_{\equiv_{h_{1,2}}}$, 
and on the right $1=[\top_0]_{\equiv_{h_{2,1}}}$ and $0=[\bot_0]_{\equiv_{h_{2,1}}}=[\bot_1]_{\equiv_{h_{2,1}}}$, $0$ corresponding to $\bot_0$ and $1$ to $\top_0$.

It is not difficult to see (for instance using the techniques in~\cite{wollic19,synt}) that we have
$\der_{R_\mathsf{u}}{=}\der_{\mathbb{U}_{1,1}}{\subsetneq} \der_{h_{1,2}[\mathbb{D}_{1,2}]}{=}\der_{R_{1,2}}$ with $R_{1,2}=\{\frac{p\,,\,q}{p\to q}\}$.

Also, we have $\der_{R_\mathsf{u}}{=}\der_{\mathbb{U}_{1,1}}{\subsetneq}\der_{h_{2,1}[\mathbb{D}_{2,1}]}{=}\der_{R_{2,1}}$ with $R_{2,1}=\{\frac{}{p\,,\,q\,,\,p\to q}\}$. Interestingly, in this case, one still has $\vdash_{\mathbb{U}_{1,1}}{=}\vdash_{h_{2,1}[\mathbb{D}_{2,1}]}$
{
as for any $\varphi\in Fm$,
$v_\varphi\in \Val(h_{2,1}[\mathbb{D}_{2,1}])$ where $$v_\varphi(\psi)=\begin{cases}
    0 &\text{if }\psi=\varphi\\
    1 &\text{otherwise }
\end{cases}$$ shows that  $v_\varphi^{-1}(1)=Fm\setminus \{\varphi\}\not\vdash_{h_{2,1}[\mathbb{D}_{2,1}]}\varphi$.} 

When $n,m>1$ we have that $h_{n,m}$ is covering, that is, $h_{n,m}[\mathbb{D}_{n,m}]$ is isomorphic 
to $\mathbb{U}_{1,1}$. Since $\der_{\mathbb{D}_{n,m}}p\to p$, we thus have $\der_{\mathbb{U}_{1,1}}{=}\der_{h_{n,m}[\mathbb{D}_{n,m}]}{\subsetneq}\der_{\mathbb{D}_{n,m}}$.  \hfill$\triangle$

\end{example}

The phenomena observed above 
suggest that it may not be possible 
to define a reasonable general notion of model reduction in the non-deterministic setting. However, we are going to see below that we can still adapt the usual constructions in order to characterize the class of Nmatrix models of a given logic.

\subsection{Sound homomorphisms and quotients} 

{

Example~\ref{smashintobool} shows that
images by strict homomorphisms (or compatible quotients) of Nmatrices 
  need not define 
the same logic, in contrast to what happens with ordinary matrices. 
In general, the image or quotient Nmatrix may define a  weaker logic than the original one. 
Therefore, when analyzing the semantics of a given ambient logic, we need to consider only those images which are  models of the logic at hand. 
{
This leads us to consider a
family of strict
homomorphisms that preserve the initial logic, that is,
whose image is still sound with respect to the initial logic. 
These
 will be sufficient}
 to recover (sound) countable Nmatrices as compatible quotients of (sound)
Lindenbaum matrices. We shall see this in the next Section.

}

\begin{definition}
 Given an ambient logic $\der$
 and a set of Nmatrices $\cM$, we denote by
 $\Quo(\cM)$ the class of isomorphic copies of $\der$-sound compatible quotients of Nmatrices in $\Mt$. That is, $\Quo(\cM)$ contains the elements of $\Nmatr(\der)$ that are images 
 of Nmatrices in $\cM$ by strict homomorphisms.
\end{definition}

The $\Quo$ operator may seem artificial, but one of the key points of the present paper is precisely to highlight the additional difficulties posed by Nmatrices, in particular the
problem of determining when the 
{
logic induced by the} image of an Nmatrix by a strict homomorphism is weaker than the ambient logic. We do not know if a  `local' characterization
of sound morphisms/quotients, analogous to congruences in logical matrices,
 is possible
at all in our more general context.

\paragraph*{Troubles with model reduction.}  

A crucial property within the theory of logical matrices is that the  congruences compatible with the designated elements of an ordinary matrix form a complete sublattice of the lattice of all congruences
of the underlying algebra~\cite[Thm.~4.20]{Fo16}. This entails that there always is a maximal element among the compatible congruences, called the \emph{Leibniz congruence}, which allows for a neat notion of reduced model~\cite[Def.~4.21]{Fo16}.

The above does not hold
in the non-deterministic setting, for we need to consider  quotients which are sound with respect to some ambient logic; and even in such a scenario, the corresponding operation of model reduction may yield multiple ``reduced models'', as we illustrate below.

\begin{example}\label{exquo}
 
Recall the examples in Subsection~\ref{sec:noprops},
and
pick $\Mt=\tuple{A_{1,3},\cdot_\Mt,D_3}$ as the subNmatrix (or refinement, in the sense of~\cite{rexpansions}) of  $\mathbb{D}_{1,3}$ where the following four entries of the corresponding truth-tables are redefined in order to contain only designated elements:
$$\to_\Mt(\top_0,\top_1)\,{=}\to_\Mt(\top_1,\top_0)\,{=}\to_\Mt(\top_1,\top_2)\,{=}\to_\Mt(\top_2,\top_1)\,{=}\,D_3.$$

It is straightforward that $\der_{R_{\mathsf{id}}}{\subseteq}\der_{\Mt}$, i.e., $\Mt$ is $\der_{R_{\mathsf{id}}}$-sound. Let $\equiv_{01}$ be the equivalence relation generated by 
identifying $\top_0{=}\top_1$, $\equiv_{12}$ be the equivalence relation generated by identifying $\top_1{=}\top_2$, and 
$\equiv$ be the equivalence relation obtained by identifying the three of them, i.e., $\top_0{=}\top_1{=}\top_2$. All are clearly compatible.

Clearly, both $\Mt_{/\equiv_{01}}$ and $\Mt_{/\equiv_{12}}$ are isomorphic to $\mathbb{D}_{1,2}$, thus the quotients are $\der_{R_{\mathsf{id}}}$-sound.

However, there is no equivalence relation extending 
$\equiv_{01}$ and $\equiv_{12}$ that is compatible with $D_3$ and also  $\der_{R_{\mathsf{id}}}$-sound.
The only possible candidate is $\equiv$, but  
$\Mt_{/\equiv}$ induces
a strictly weaker logic than $\der_{R_{\mathsf{id}}}$. In particular, we have
$\not\der_{\Mt_{/\equiv}}p\to p$ 
since $[\bot_0]_{\equiv}\,{\in} \to_{\Mt/\equiv}([\top_0]_\equiv,[\top_0]_\equiv)$.
In summary, $\Mt_\equiv$ is not $\der_{R_{\mathsf{id}}}$-sound.
\hfill $\triangle$
\end{example}

\paragraph*{Countable Nmatrices are still quotients of Lindenbaum matrices.} Despite the differences stressed above, we can still collect all sound compatible quotients of Nmatrix models 
-- instead of simply considering reduced models, as one does with ordinary matrices -- 
and obtain interesting characterizations. Recall the set $\cL_\der$ of all Lindenbaum matrices of a logic  introduced in Definition~\ref{def:LindMatr}.\smallskip

{
For a given $\Sigma$-Nmatrix $\Mt=\tuple{A,\cdot_\Mt,D}$, let   
$\mathsf{Equiv}(\Mt)$ denote the set of all compatible
equivalence relations over $A$.

\begin{definition}
Given a logic $\der$, define the class of all $\der$-sound compatible quotients of Lindenbaum matrices as
$\cL^{\mathsf{quo}}_\der=\{\Mt_{/\equiv}:\Mt\in \cL_\der,\equiv\, \in \mathsf{Equiv}(\Mt),\der{\subseteq} \der_{\Mt_{/\equiv}}\}$.  
\end{definition}

A fundamental fact is that every countable $\der$-sound Nmatrix can be obtained as 
a $\der$-sound compatible quotient of a Lindenbaum matrix. 
Given a valuation $v\in\Val(\Mt)$ over a $\Sigma$-Nmatrix $\Mt$,  
let the equivalence relation $\equiv_v{\subseteq}\, Fm\times Fm$ be the 
kernel 
given by $\varphi\equiv_v\psi$ whenever $v(\varphi)=v(\psi)$,
and set $D_v=v^{-1}(D)$.
 
\begin{proposition}\label{breakintoden}
For every countable 
 $\Sigma$-Nmatrix 
 $\Mt\in \Nmatr(\der)$
 there is $v\in \Val(\Mt)$ such that 
  $(\mathbb{L}_v)_{/\equiv_v}$ is
 isomorphic to $\Mt$.
 
 As $\mathbb{L}_v\in  \cL_\der$ and $\equiv_v$ is compatible, we have $\Mt\in \Quo(\cL_\der)$.
\end{proposition}

\begin{proof}
Let $\Mt=\tuple{A,\cdot_\Mt,D}$ be a $\der$-sound countable Nmatrix. For each $k<\omega$ define
$$B_k=\{(\conn,x_1,\dots,x_k,y):\conn\in\Sigma^k,x_1,\dots,x_k\in A,y\in\conn_\Mt(x_1,\dots,x_k)\}\textrm{, and}$$
$$\widetilde{B}=A\cup\bigcup\limits_{k<\omega}\{(b,i):b\in B_k,1\leq i\leq k\}.$$

It is clear that $\widetilde{B}$ is countable since $\Sigma$ and $A$ are countable, so we can consider an injective function $e:\widetilde{B}\to \nats$. 
For each $x\in A$, let $\varphi_x$ denote the propositional variable $p_{e(x)}$.
For each $k<\omega$ and $b=(\conn,x_1,\dots,x_k,y)\in B_k$, let $\varphi_b$ denote the formula
$\conn(p_{e(b,1)},\dots,p_{e(b,k)})$. Finally, define
$$\Gamma=\sub(\{\varphi_x:x\in A\}\cup\bigcup\limits_{k<\omega}\{\varphi_b:b\in B_k\})$$
and consider the function $w:\Gamma\to A$ such that $w(\varphi_x)=x$ for each $x\in A$, and also $w(p_{e(b,i)})=x_i$ and $w(\varphi_b)=y$ for each $k<w$, 
$b=(\conn,x_1,\dots,x_k,y)\in B_k$ and $1\leq i\leq k$. It is clear that $w$ is a prevaluation and thus that it can be extended to a valuation $v\in\Val(\Mt)$.

Clearly, $v:\mathbb{L}_v\to\Mt$ is a strict homomorphism. Further, $v$ is onto  because $v(\varphi_x)=x$ for each $x\in A$. To see that $v$ is covering note that if $y\in\conn_\Mt(x_1,\dots,x_k)$ then one can consider $b=(\conn,x_1,\dots,x_k,y)\in B_k$ and observe that $v(p_{e(b,i)})=x_i$ for $1\leq i\leq k$, and that $v(\varphi_b)=v(\conn(p_{e(b,1)},\dots,p_{e(b,k)}))=y$. 

We conclude that ${\mathbb{L}_v}_{/\equiv_v}$ is
 isomorphic to $\Mt$.
\end{proof}

\paragraph*{Comparing logics given by denumerable Nmatrices.} Although the problem of comparing the logics defined by two Nmatrices seems to be inherently harder than in the case of logical matrices (as we discuss in the concluding section), we can still obtain the following characterization.

\begin{theorem}\label{th:countablereflect}
Given  two countable $\Sigma$-Nmatrices $\Mt_1$ and $\Mt_2$,  we have $\der_{\Mt_1}{\subseteq} \der_{\Mt_2}$ if and only if
  $\Mt_2$ is a $\der_{\Mt_1}$-sound quotient of a preimage of $\Mt_1$.
Hence, $\der_{\Mt_1}{=} \der_{\Mt_2}$ if and only if  $\Mt_1\in \QuoMd({\Rexp(\SubN}(\Mt_2)))$
and $\Mt_2\in \QuoMu({\Rexp(\SubN}(\Mt_1)))$.
\end{theorem}
\begin{proof}
    Immediate from Lemmas~\ref{valrexp} and \ref{lem:rexpstronger}, and Proposition~\ref{breakintoden}.    
\end{proof}

The following example illustrates the result.

\begin{example}\label{ex3val}

Recall the examples of Subsection~\ref{sec:noprops}. Let us consider a signature with a single unary connective $\neg$, and the Nmatrices 
 $\Mt_1=\tuple{A_{1,2},\cdot_1,D_2}$,  $\Mt_2=\tuple{A_{1,2},\cdot_2,D_2}$, 
 $\Mt_3=\tuple{A_{1,2},\cdot_3,D_2}$, $\Mt_4=\tuple{A_{1,3},\cdot_4,D_3}$
 where the different interpretations of $\neg$ are given by the following tables:
\begin{center}
    \begin{tabular}{c | c}
                & $\neg_1(x)$ \\
                \hline
                $\bot_0$ & $\top_0$ \\
                $\top_0$ & $\bot_0$ \\
                $\top_1$ & $\bot_0,\top_1$ 
                  \end{tabular}
                  \qquad
                   \begin{tabular}{c | c}
                & $\neg_2(x)$ \\
                \hline
              $\bot_0$ & $\top_0$ \\
                $\top_0$ & $\bot_0$ \\
                $\top_1$ & $\top_0,\top_1$ 
                  \end{tabular}
               \qquad    \begin{tabular}{c | c}
                & $\neg_3(x)$ \\
                \hline
               $\bot_0$ & $\top_0$ \\
                $\top_0$ & $\bot_0$ \\
                $\top_1$ & $\bot_0,\top_0,\top_1$ 
                  \end{tabular}
            \qquad             \begin{tabular}{c | c}
                & $\neg_4(x)$ \\
                \hline
               $\bot_0$ & $\top_0$ \\
                $\top_0$ & $\bot_0$ \\
                $\top_1$ & $\bot_0,\top_1$ \\
                $\top_2$ & $\top_0,\top_1$ 
            \end{tabular}
\end{center}

With $R=\{\frac{}{p\,,\,\neg p}, \frac{\neg\neg p}{p}\}$, it is easy to see that $\der_R{\subseteq} \der_{\Mt_i}$ for $1\leq i\leq 4$. Furthermore, as $\neg_2(\top_0)=\{\bot_0\}\subseteq U_1$ and 
$\neg_2(\top_1)=\{\top_0,\top_1\}\subseteq D_2$, $\Mt_2$ is \emph{monadic}\index{monadic Nmatrix} and we 
 obtain that $\der_{\Mt_2}{=}\der_R$ using the techniques in~\cite{wollic}.

We have that $\Mt_1$ is subNmatrix of $\Mt_3$ and $\Mt_4$, $\Mt_2$ is subNmatrix of $\Mt_3$, and thus $\Mt_1,\Mt_2\in  {\SubN}(\Mt_3)$ and $\Mt_1\in  {\SubN}(\Mt_4)$.
Further, $\Mt_3$ is isomorphic to the quotient of $\Mt_4$ obtained by identifying $\top_1=\top_2$. 
We obtain by Lemma~\ref{lem:rexpstronger} that 
 $\der_{\Mt_2}{=}\der_{R} {\subseteq} \der_{\Mt_3}{\subseteq}\der_{\Mt_4}{\subseteq} \der_{\Mt_1}$.

Let us further see that $\der_{\Mt_1}{\subseteq} \der_{\Mt_2}$, and thus conclude that all the Nmatrices $\Mt_i$ induce the same multiple-conclusion logic.
Using Theorem~\ref{th:countablereflect}, it is enough to show that
  $\Mt_1\in \QuoMd(\Rexp(\Mt_2))\subseteq \QuoMd(\Rexp(\SubN(\Mt_2)))$.
Consider $\Mt=\tuple{A,\cdot_\Mt,D}$
whose truth values are the sequences $xy$ with $x,y\in A_{1,2}$ and $x\in\neg_2(y)$ ($y$ is a look-behind of $x$ in $\Mt_2$),
i.e., $A=\{\bot_0\top_0,\top_0\bot_0,\top_0\top_1,\top_1\top_1\}$,
 the designated elements are 
$D=\{xy\in A: x\in D_2\}=\{\top_0\bot_0,\top_0\top_1,\top_1\top_1\}$,
$\neg_\Mt(\top_0\bot_0)=\neg_\Mt(\top_0\top_1)=\{\bot_0\top_0\}$ and
$\neg_\Mt(\bot_0\top_0)=\{\top_0\bot_0\}$ and $\neg_\Mt(\top_1\top_1)=\{\top_0\top_1,\top_1\top_1\}$.
The function $h:A\to A_{1,2}$ given by
$h(xy)=x$ defines a { covering} strict homomorphism $h:\Mt\to\Mt_2$, and 
$\Mt\in \Rexp(\Mt_2)$.
Finally, consider $g:A\to A_{1,2}$ given by
 $g(\bot_0\bot_0)=\bot_0$, $g(\top_0\bot_0)=\top_0$ and
$g(\top_0\top_1)=h'(\top_1\top_1)=\top_1$.
We have that $g$
 defines a covering strict homomorphism $g:\Mt\to\Mt_1$, and thus $\Mt_1\in \QuoMd(\Mt)$
 since we already know that $\der_{\Mt_2}{\subseteq}\der_{\Mt_1}$.

 We conclude that $\der_{\Mt_1}{=}\der_{\Mt_2}{=}\der_{\Mt_3}{=}\der_{\Mt_4}{=}\der_{R}$, and by Theorem~\ref{th:countablereflect} we get that
$\Mt_i\in \mathsf{Hom}_s^{\der_{\Mt_j}}\Rexp(\Mt_j)$ for every $1\leq i, j\leq 4$.
{
For instance, to see that $\Mt_2\in \QuoMu(\Rexp(\Mt_1))$,
consider now the Nmatrix $\Mt=\tuple{A,\cdot_\Mt,D}$
whose truth values are the sequences $xy$ with $x,y\in A_{1,2}$ and $y\in\neg_1(x)$ ($y$ is a look-ahead of $x$ in $\Mt_1$), i.e., 
$A=\{\bot_0\top_0,\top_0\bot_0,\top_1\bot_0,\top_1\top_1\}$, 
$D=\{\top_0\bot_0,\top_1\bot_0,\top_1\top_1\}$,
and
$\neg_\Mt(\bot_0\top_0)=\{\top_0\bot_0\}$, 
$\neg_\Mt(\top_0\bot_0)=\neg_\Mt(\top_1\bot_0)=\{\bot_0\top_0\}$ and
$\neg_\Mt(\top_1\top_1)=\{\top_1\bot_0,\top_1\top_1\}$. 
As before, the function $h:A\to A_{1,2}$ given by
$h(xy)=x$ defines a { covering} strict homomorphism $h:\Mt\to\Mt_1$, yielding $\Mt\in \Rexp(\Mt_1)$. Further, the function $g:A\to A_{1,2}$ given by
$g(\bot_0\top_0)=\bot_0$, 
$g(\top_0\bot_0)=h_1(\top_1\bot_0)=\top_0$ 
and 
 $g(\top_1\top_1)=\top_1$,  defines a covering strict homomorphism
 $g:\Mt\to\Mt_2$.
}
\hfill$\triangle$
\end{example}

\subsection{  
Products and ultraproducts} 
Let us now look at products and ultraproducts of Nmatrices: this
will allow us to deal with non-denumerable models and also to better characterize single-conclusion logics.

\begin{definition}

 Let $I$ be a set, and $\Mt_i=\tuple{A_i,\cdot_i,D_i}$ with $i\in I$ be a family of $\Sigma$-Nmatrices. 
 The \emph{product}\index{product of Nmatrices} $\Pi_{i\in I} \Mt_i$ is the Nmatrix $\tuple{\Pi_{i\in I}A_i,\cdot_I,\Pi_{i\in I}D_i}$
 where 
 $$\conn_I(s_1,\ldots,s_k)=\Pi_{i\in I}\conn_i(\pi_i(s_1),\ldots,\pi_i(s_k)),$$
 where $\pi_i:\Pi_{i\in I}A_i\to A_i$ is the corresponding projection function for each $i\in I$.
  \end{definition}

 Note that if every $\Mt_i$ is deterministic then so is $\Pi_{i\in I}\Mt_i$.  
 Given a class $\cM$ of $\Sigma$-Nmatrices let $\Pi(\cM)$ denote the class of all isomorphic copies of products of families contained in $\cM$.\\

It is clear that $v\in\Val(\Pi_{i\in I}\Mt_i)$ if and only if $\pi_i\circ v\in\Val(\Mt_i)$ for every $i\in I$. Hence, given a family of valuations $v_i\in\Val(\Mt_i)$ for each $i\in I$, there is a unique $v\in\Val(\Pi_{i\in I}\Mt_i)$ such that $v_i=\pi_i\circ v$ for every $i\in I$.

\paragraph*{Single-conclusion logics viewed as multiple-conclusion.} Throughout  the present paper, we have been essentially concerned with multiple-conclusion logics. Unsurprisingly, the usual notion of single-conclusion logic can be characterized within the more general context. Indeed, a single-conclusion logic $\vdash$ can be recast as (the single-conclusion companion $\vdash_\der$ of) a multiple-conclusion logic $\der$ whose models are closed under arbitrary products.
{
By picking $\der{=}\der_\vdash$, we precisely obtain
$\Nmatr(\der)=\Nmatr(\der_\vdash)=\Nmatr(\vdash)$. The next lemma shows that the closure under products of any adequate semantics for $\vdash$ is also adequate for $\der_\vdash$.
}
 
\begin{lemma}
For any class $\cM$ of $\Sigma$-Nmatrices, we have that
 $\vdash_\cM{=}\vdash_{\Pi(\cM)}$
 and 
 $\der_{\vdash_\cM}{=}\der_{\Pi(\cM)}$.
\end{lemma}
\begin{proof}
From $\cM\subseteq \Pi(\cM)$ it
 follows that 
$\vdash_{\Pi(\cM)}{\subseteq}\vdash_\cM$.
For the converse inclusion, let $\Gamma\not\vdash_{\Pi(\cM)} \varphi$.
Then there is a valuation $v$ over  $\Pi_{i\in I}\Mt_i$ with $\Mt_i\in\cM$ for $i\in I$ such that
$v[\Gamma]\subseteq \Pi_{i\in I}D_i$ and $v(\varphi)\notin \Pi_{i\in I}D_i$.
Hence, there is $i\in I$ such that
$(\pi_i\circ v)[\Gamma]\subseteq D_i$
and $(\pi_i\circ v)(\varphi)\notin D_i$. Thus, we have $\vdash_\cM{=}\vdash_{\Pi(\cM)}$.

 By definition, 
 $\Gamma\not\der_{\vdash_\cM}\Delta$ if and only if $\Gamma\not\vdash_\cM \varphi$ for each
 $\varphi\in \Delta$.
Hence, for each $\varphi\in \Delta$
there is a valuation $v_\varphi\in\Val(\Mt_\varphi)$ on $\Mt_\varphi=\tuple{A_\varphi,\cdot_\varphi,D_\varphi}\in \cM$
such that $v_\varphi[\Gamma]\subseteq D_\varphi$ and 
 $v_\varphi(\varphi)\notin D_\varphi$.
But this is equivalent to having a valuation $v\in\Val(\Pi_{\varphi\in\Delta}\Mt_\varphi)$ such that $v[\Gamma]\subseteq \Pi_{\varphi\in\Delta}D_\varphi$ and $v[\Delta]\cap \Pi_{\varphi\in\Delta}D_\varphi=\emptyset$, that is, 
$\Gamma\not\der_{\Pi(\cM)}\Delta$. We conclude that $\der_{\vdash_\cM}{=} \der_{\Pi(\cM)}$.
\end{proof}

\paragraph*{Ultraproducts.} We are now going to see that ultraproducts of Nmatrices can  be used to characterize compactness  in the multiple-conclusion setting too.

 \begin{definition}  
 
 Let $I$ be a set, $\Uf$ an ultrafilter\footnote{Recall that an ultrafilter on $I$ is a set $\Uf\subseteq \wp(I)$ such that $\emptyset\not\in \Uf$,
 if $X\subseteq Y\subseteq I$ and $X\in \Uf$ then $Y\in \Uf$, 
 if $X,Y\in \Uf$ then $X\cap Y\in \Uf$, and for every $X\subseteq I$ either $X\in \Uf$ or $I\setminus X\in \Uf$. } on $I$, and $\Mt_i=\tuple{A_i,\cdot_i,D_i}$ with $i\in I$ be a family of $\Sigma$-Nmatrices. The \emph{ultraproduct}\index{ultraproduct of Nmatrices} $\Pi_\Uf \Mt_i$ is the quotient Nmatrix 
 $(\Pi_{i\in I}\Mt_i)_{/\!\equiv_\Uf}$ where
 $s\equiv_\Uf t$ whenever $\{i\in I:\pi_i(s)=\pi_i(t)\}\in \Uf$.

 Given a class $\cM$ of Nmatrices  we denote by $\Up(\cM)$
the class of isomorphic copies of ultraproducts of  families of Nmatrices in $\cM$.
\end{definition}

It is straightforward to check that the above-defined relation $\equiv_\Uf$ is an equivalence on $\Pi_{i\in}A_i$. 
To simplify the notation,
we denote by $[s]_\Uf$ (instead of $[s]_{\equiv_\Uf}$) the equivalence class of each 
$s\in\Pi_{i\in}A_i$.

As one would expect, $[s]_\Uf$ is designated in $\Pi_\Uf \Mt_i$ precisely if $\{i\in I:\pi_i(s)\in D_i\}\in\Uf$. The following result shows that the interpretation of connectives in ultraproducts of Nmatrices also behaves as expected. 

\begin{lemma}\label{lem:up}

Let $I$ be a set, $\Uf$ an ultrafilter on $I$, and $\Mt_i=\tuple{A_i,\cdot_i,D_i}$ with $i\in I$ a family of $\Sigma$-Nmatrices whose {ultraproduct} $\Pi_\Uf \Mt_i=\tuple{(\Pi_{i\in I}A_i)_{/\!\equiv_\Uf},\cdot_\Uf,(\Pi_{i\in I}D_i)_{/\!\equiv_\Uf}}$. Then, for $k<\omega$ and $\conn\in\Sigma^k$, we have that $[s]_\Uf\in  \conn_{\equiv_\Uf}([s_1]_\Uf,\ldots,[s_k]_\Uf)$ 
if and only if $X=\{i\in I:\pi_i(s)\in \conn_{i}(\pi_i(s_1),\ldots,\pi_i(s_k))\}\in \Uf$.

 \end{lemma}
\begin{proof}
If $[s]_\Uf\in  \conn_{\equiv_\Uf}([s_1]_\Uf,\ldots,[s_k]_\Uf)$ then there are $r\equiv_\Uf s$, $r_1\equiv_\Uf s_1$, $\dots$, $r_k\equiv_\Uf s_k$ such that $r\in\conn_{I}(r_1,\dots,r_k)$, that is, $\pi_i(r)\in\conn_i(\pi_i(r_1),\dots,\pi_i(r_k))$ for every $i\in I$. By definition of $\equiv_\Uf$ we have that $Y=\{i\in I:\pi_i(r)=\pi_i(s)\}\in\Uf$, and $Y_j=\{i\in I:\pi_i(r_j)=\pi_i(s_j)\}\in\Uf$ for $1\leq j\leq k$. Hence, $Z=Y\cap Y_1\cap\dots\cap Y_k\in\Uf$ and $Z=\{i\in I:\pi_i(r)=\pi_i(s),\pi_i(r_1)=\pi_i(s_1),\dots,\pi_i(r_k)=\pi_i(s_k)\}\subseteq X$ which guarantees that $X\in\Uf$.

Reciprocally, knowing $X\in\Uf$, take any $r\in\conn_I(s_1,\dots,s_k)$ with $\pi_i(r)=\pi_i(s)$ for all $i\in X$. We have $r\equiv_\Uf s$ and thus $[s]_\Uf=[r]_\Uf\in\conn_{\equiv_\Uf}([s_1]_\Uf,\ldots,[s_k]_\Uf)$.
\end{proof}

The following result   
shows that we are working with
 a smooth generalization of the notion of ultraproduct to the non-deterministic environment.

\begin{proposition}\label{UPdet}
Ultraproducts preserve determinism, that is, 
if $\Mt_i$ is a deterministic matrix for all $i\in I$, then
so is $\Pi_\Uf \Mt_i$.
 \end{proposition}
\begin{proof}
Suppose that $[s]_\Uf,[r]_\Uf\in  \conn_{\equiv_\Uf}([s_1]_\Uf,\ldots,[s_k]_\Uf)$. From Lemma~\ref{lem:up} we know that $X_s=\{i\in I:\pi_i(s)\in \conn_{\Mt_i}(\pi_i(s_1),\ldots,\pi_i(s_k))\}\in \Uf$ and $X_r=\{i\in I:\pi_i(r)\in \conn_{\Mt_i}(\pi_i(s_1),\ldots,\pi_i(s_k))\}\in \Uf$. Hence, 
$Y=X_s\cap X_r=\{i\in I:\pi_i(s),\pi_i(r)\in \conn_{\Mt_i}(\pi_i(s_1),\ldots,\pi_i(s_k))\}\in \Uf$, 
and since each $\Mt_i$ is deterministic it follows that $Y\subseteq \{i\in I:\pi_i(s)=\pi_i(r)\}\in\Uf$, which guarantees that $s\equiv_\Uf r$ and $[s]_\Uf=[r]_\Uf$.
\end{proof}

Unsurprisingly, valuations on ultraproducts can also be viewed as ultraproducts of valuations. Our construction, however, highlights the fact that, due to non-determinism, valuations are no longer determined by their values on propositional variables. 

\begin{proposition}\label{prop:decomp}
For every valuation $v\in\Val(\Pi_{\Uf}\Mt_i)$ there is a family of valuations $v_i\in\Val(\Mt_i)$ for $i\in I$ such that, for every formula $\varphi\in Fm$, exists $s_\varphi\in \Pi_iA_i$ such that
$v(\varphi)=[s_\varphi]_\Uf$ with $\pi_i(s_\varphi)=v_i(\varphi)$ for each $i\in I$.
\end{proposition}
\begin{proof}
Let
  $\Theta_n=\{\varphi:
  \mathsf{depth}(\varphi)\leq n\}$ for  $n\geq 0$.
   Note that each $\Theta_n$ is  closed under subformulas.
  For each $i\in I$, we will build $v_i$ inductively, as the limit  of a sequence of prevaluations $v_i^n$ defined on 
  $\Theta_n$, where $v_i^n$ is the restriction of $v_i^{n+1}$ to $\Theta_{n}$.

 The value of each $v_i$ for variables $p\in P=\Theta_0$ is set at step $n=0$,
by picking any $s_{p}$ such that $v(p)=[s_{p}]_\Uf$, and defining $v^0_i(p)=\pi_i(s_{p})$.

In step $n+1>0$, 
for every $\psi\in \Theta_n$, we set $v^{n+1}_i(\psi)=v^n_i(\psi)$. 
By induction hypothesis,
$v^{n+1}_i(\psi)=\pi_i(s_\psi)$ satisfies the requirements regarding formulas 
$\psi\in\Theta_{n}$.
For  $\psi=\conn(\psi_1,\ldots,\psi_k)\in \Theta_{n+1}\setminus\Theta_{n}$,
we surely have that $\psi_j\in \Theta_n$ for $1\leq j\leq k$.
Thus, for some $r_\psi\in \Pi_{i\in I}A_i$ $v(\psi)=[r_\psi]_\Uf\in \conn_{\equiv_\Uf}(v(\psi_1),\ldots,v(\psi_k))=\conn_{\equiv_\Uf}([s_{\psi_1}]_\Uf,\ldots,[s_{\psi_k}]_\Uf)$.
Note that, by Lemma~\ref{lem:up}, we know that 
$X=\{i\in I:\pi_i(r_\psi)\in \conn_{\Mt_i}(\pi_i(s_{\psi_1}),\ldots,\pi_i(s_{\psi_k}))\}\in \Uf$.
Consider any $s_\psi\in\conn_I(s_{\psi_1},\ldots,s_{\psi_k})$ such that $\pi_i(s_\psi)=\pi_i(r_\psi)$ for all $i\in X$.
We have $r_\psi\equiv_\Uf s_\psi$.
By setting $v_i^{n+1}(\psi)=\pi_i(s_\psi)$ we 
guarantee that the $v^{n+1}_i$ are prevaluations on $\Mt_i$ defined over $\Theta_{n+1}$,
and that for every $\varphi\in \Theta_{n+1}$
 we have 
 $v(\varphi)=[s_\varphi]_\Uf=[r_\varphi]_\Uf$ where $\pi_i(s_\varphi)=v^{n+1}_i(\varphi)$ for each $i\in I$.
Picking for each $i\in I$, $v_i(\varphi)=v^{n}_i(\varphi)$ whenever $\varphi\in \Theta_n$ terminates the proof.
\end{proof}

\begin{proposition}\label{valsUP}
Given valuations $v_i\in\Val(\Mt_i)$ for $i\in I$
and an ultrafilter $\Uf$ over $I$,
there is an unique $v\in \Val(\Pi_\Uf\Mt_i)$ such that for each $\varphi\in Fm$ there exists $s_\varphi\in \Pi_{i\in I}A_i$ such that
$v(\varphi)=[s_\varphi]_\Uf$ and $\pi_i(s_\varphi)=v_i(\varphi)$.
\end{proposition}
\begin{proof}
Taking $s_\varphi$ defined by $\pi_i(s_\varphi)=v_i(\varphi)$ and making
$v(\varphi)=[s_\varphi]_\Uf$ we have that $v\in \Val(\Pi_{\Uf}\Mt_i)$ since  $I\in \Uf$.
Uniqueness follows from the fact that the $v_i$'s completely define $s_\varphi$ and therefore $v$.
\end{proof}

\paragraph*{Recovering an Nmatrix from its countable parts.}
 
Since we always work with countable languages, it is clear that countable subNmatrices of models are enough to characterize a logic.
Let $\cM$ be the set of all (non-empty) countable subNmatrices of an arbitrary $\Sigma$-Nmatrix $\Mt$. Then we  have
$\der_\Mt{=}\der_\cM$ and $\Mt$ can always be embedded in an ultraproduct of $\cM$.
\begin{proposition}\label{prop:denum}
 Let $\cM=\{\Mt_i:i\in I\}$ be the set of all (non-empty) countable subNmatrices of $\Mt$.
 There is 
 an injective strict homomorphism 
 $h:\Mt\to \Pi_{\Uf} \Mt_i$ for some ultrafilter $\Uf$ over $I$. Thus, 
 $\Mt\in \SubN(\Up(\cM))$.  
\end{proposition}
\begin{proof}
Let $\Mt=\tuple{A,\cdot_\Mt,D}$ and $\Mt_i=\tuple{A_i,\cdot_i,D_i}$ for $i\in I$.
Consider for each $a\in A$ the set 
   $X_a=\{i\in I:a \in A_i\}$.
Given any finite $B\subseteq A$ we have that 
$X_B=\bigcap_{a\in B}X_a\neq\emptyset$, since there are always countable subNmatrices containing finitely many
  elements.
Then $F=\{X_B: B\in\wp_{\mathsf{fin}}(A)\}$ 
  is a filterbase. Let $\Uf$ be any ultrafilter extending the filter generated by $F$.

  We fix $a_i\in A_i$ for every $i\in I$, and for each $a\in A$ we consider $s_a\in \Pi_{i\in I}A_i$ given by
  $$\pi_i(s_a)=
\begin{cases}
 a \mbox{ if }a\in A_i\\
 a_i\mbox{ otherwise.}
\end{cases}
$$
and the function
  $h:A\to (\Pi_{i\in I}A_i)_{/\!\equiv_\Uf}$ given by
   $h(a)=[s_a]_\Uf$.
 Clearly, we have $X_a\subseteq Y_a=\{i\in I:\pi_i(s_a)=a\}\in\Uf$. We show that $h$ defines an injective strict homomorphism.

Suppose that $a,b\in A$ are such that
$h(a)=h(b)=[s_a]_\Uf=[s_b]_\Uf$. We have $Y_a,Y_b\in\Uf$. Further, as $s_a\equiv_\Uf s_b$, we also have $Y=\{i\in I:\pi_i(s_a)=\pi_i(s_b)\}\in\Uf$. 
Thus $\emptyset\neq Y\cap Y_a\cap Y_b\in\Uf$, so there exists $i\in I$ with $\pi_i(s_a)=a=b=\pi_i(s_b)$, and $h$ is injective.

Suppose that $a\in \conn_\Mt(a_1,\ldots,a_k)$. To see that $h(a)\in\conn_{\equiv_\Uf}(h(a_1),\dots,h(a_k))$, that is, $[s_a]_\Uf\in\conn_{\equiv_\Uf}([s_{a_1}]_\Uf,\dots,[s_{a_k}]_\Uf)$ it suffices, using Lemma~\ref{lem:up}, to observe that 
$Y_a\cap Y_{a_1}\cap\dots\cap Y_{a_k}\subseteq X=\{i\in I:\pi_i(s_a)\in\conn_i(\pi_i(s_{a_1}),\dots,\pi_i(s_{a_k}))\}\in\Uf$. 

In order to conclude that $h$ is a strict homomorphism we still need to check that $h^{-1}((\Pi_{i\in I}D_i)_{/\!\equiv_\Uf})=D$.
Given $a\in A$, let $Y=\{i\in I:\pi_i(s_a)\in D_i\}$.
It is clear that if $a\in D$ then
$X_a\subseteq Y\in \Uf$ and
therefore $h(a)\in  (\Pi_{i\in I}D_i)_{/\!\equiv_\Uf}$.
Reciprocally, 
if $h(a)\in  (\Pi_{i\in I}D_i)_{/\!\equiv_\Uf}$ we get that 
$Y\in \Uf$ and $\emptyset\neq X_a\cap Y\in \Uf$, so there exists $i\in I$ such that $\pi_i(s_a)=a\in D_i$, and thus $a\in D$.
\end{proof}

\section{Characterizations}\label{sec:results}

In this section  we characterize   key  properties of a logic $\der$ in terms of closure properties of its non-deterministic models, partly extending the results obtained in Czelakowski's~\cite{czela1983}.

\subsection{Characterizing compact logics}

We begin by verifying that the ultraproduct of $\der$-sound Nmatrices is $\der$-sound whenever 
$\der$ is compact.

\begin{proposition}\label{soundup}
For every
compact logic $\der$, the class
$\Nmatr(\der)$ is closed under ultraproducts.
\end{proposition}
\begin{proof}
Given
set of Nmatrices
 $\cM=\{\Mt_i:i\in I\}\subseteq\Nmatr(\der)$ and
 $\Uf$ an ultrafilter on $I$, we will show that
 $\Pi_\Uf \Mt_i\in\Nmatr(\der)$.

By compactness, if $\Gamma\der \Delta$ we know that there are finite
 $\Gamma_0\subseteq \Gamma$ and $\Delta_0\subseteq \Delta$ such that $\Gamma_0\der\Delta_0$.
 For every valuation $v\in\Val(\Pi_\Uf \Mt_i)$ 
 we know by Proposition~\ref{prop:decomp} that there 
 is a family of valuations $v_i\in\Val(\Mt_i)$ 
 such that
 for every formula $\varphi$ we have
 $v(\varphi)=[s_\varphi]_\Uf$ where $\pi_i(s_\varphi)=v_i(\varphi)$ for all $i\in I$.
 
Assuming that  $v[\Gamma_0]\subseteq (\Pi_{i\in I}D_i)_{/\!\equiv_\Uf}$, we get that 
$X_\varphi=\{i\in I:v_i(\varphi)\in D_i\}\in \Uf$ 
for every $\varphi\in \Gamma_0$, 
and thus also $X=\bigcap_{\varphi\in\Gamma_0}X_\varphi\in\Uf$. 
Let $Y=\{i\in I:v[\Delta_0]\cap D_i=\emptyset\}$.
Since each $\Mt_i\in\Nmatr(\der)$, we have $X\cap Y=\emptyset$ and so $Y\notin\Uf$. Note that $Y=\bigcap_{\psi\in\Delta_0}Y_\psi$ 
with $Y_\psi=\{i\in I:v_i(\psi)\notin D_i\}$, and therefore there must exist $\psi\in\Delta_0$ such that 
$Y_\psi\notin\Uf$, and thus $I\setminus Y_\psi=\{i\in I:v_i(\psi)\in D_i\}\in\Uf$. We conclude that $v(\psi)\in(\Pi_{i\in I}D_i)_{/\!\equiv_\Uf}$, $v[\Delta_0]\cap (\Pi_{i\in I}D_i)_{/\!\equiv_\Uf}\neq\emptyset$, and therefore $\Gamma_0\der_{\Pi_\Uf\Mt_i}\Delta_0$ and $\Gamma\der_{\Pi_\Uf\Mt_i}\Delta$.
 \end{proof}

  We now consider the converse of the above property.

\begin{theorem}\label{charcompact}
 Any logic $\der$  is compact if and only if $\Nmatr(\der)$ is closed under ultraproducts.
\end{theorem}
\begin{proof}
We already know from Proposition~\ref{soundup} that $\Nmatr(\der)$ is closed under ultraproducts when $\der$ is compact. For the converse, let us assume that $\der$ is not compact.
 This means that there are infinite sets $\Gamma,\Delta\subseteq Fm$ such that $\Gamma\der\Delta$, but
 $\Gamma_0\not\der\Delta_0$ for all finite $\Gamma_0\subseteq \Gamma$, 
 $\Delta_0\subseteq \Delta$. In particular, we have that $\Gamma\cap \Delta=\emptyset$.
 
 Consider $I=\wp_{\mathsf{fin}}(\Gamma)\times \wp_{\mathsf{fin}}(\Delta)$.
 By ease of notation, we denote the elements of $I$ by $i=\tuple{\Gamma_i,\Delta_i}$.
 For every
 $i=\tuple{\Gamma_i,\Delta_i}\in
 I$ there are $\Mt_i=\tuple{A_i,\cdot_i,D_i}\in \Nmatr(\der)$ and $v_i\in\Val(\Mt_i)$ such that
  $v_i[\Gamma_i]\subseteq D_i$ and $v_i[\Delta_i]\cap D_i=\emptyset$.

For each finite
$\Theta\subseteq \Gamma\cup\Delta$, let $X_\Theta=\{i\in I: \Theta\subseteq \Gamma_i\cup\Delta_i\}$. Note that $X_\Theta=\bigcap_{\psi\in\Theta}X_{\{\psi\}}$. Clearly, we have 
$\tuple{\Theta\cap \Gamma,\Theta\cap \Delta}\in X_{\Theta}\neq \emptyset$, and hence
$F=\{X_\Theta:\Theta\in\wp_{\mathsf{fin}}(\Gamma\cup\Delta)\}$ is a filterbase. Let $\Uf$
 be any
ultrafilter extending the filter generated by $F$.

  By Proposition~\ref{valsUP}, we can pick $v\in \Val(\Pi_\Uf \Mt_i)$
 with $v(\varphi)=[s_\varphi]_\Uf$ and $\pi_i(s_\varphi)=v_i(\varphi)$ for all $\varphi\in Fm$.
For each $\psi\in \Gamma$, we have that $X_{\{\psi\}}\subseteq\{i\in I:v_i(\psi)\in D_i\}\in \Uf$, 
and 
thus $v[\Gamma]\subseteq (\Pi_{i\in I}D_i)_{/\!\equiv_\Uf}$.
Furthermore, 
for every $\psi\in \Delta$, 
we have that 
$X_{\{\psi\}}\subseteq \{i\in I:v_i(\psi)\notin D_i\}\in \Uf$, and 
thus $I\setminus \{i\in I:v_i(\psi)\notin D_i\}=\{i\in I:v_i(\psi)\in D_i\}\notin\Uf$, which implies that
$v[\Delta]\cap (\Pi_{i\in I}D_i)_{/\!\equiv_\Uf}=\emptyset$.
Therefore, $\Gamma\not\der_{ \Pi_\Uf \Mt_i}\Delta$ and $ \Pi_\Uf \Mt_i\notin \Nmatr(\der)$.
\end{proof}

\subsection{Characterizing finitely based logics}

With a little extra work one can also characterize finitely based logics.

\begin{theorem}\label{charfinbased}
 
A logic  $\der$ is finitely based  if and only if $\Nmatr(\der)$ and its complement are both closed under ultraproducts.
\end{theorem}
\begin{proof}
It is well known that $\der$ is compact if and only if it is axiomatized by a (possibly infinite) set of finite rules~ \cite{Woj,ShoesmithSmiley}.
Hence, in light
of Theorem~\ref{charcompact}, it is enough to show, under the assumption that $\der$ is compact and therefore that
$\Nmatr(\der)$ is closed under ultraproducts, that $\der$ is 
finitely based if and only if the complement of $\Nmatr(\der)$ is also closed under ultraproducts.\smallskip

First assume that $\der$ is finitely based, with $\der{=}\der_R$ for some finite set of finite rules $R=\{\frac{\Gamma_j}{\Delta_j}:1\leq j\leq k\}$. Given any family   $\Mt_i=\tuple{A_i,\cdot_i,D_i}\notin \Nmatr(\der)$ for $i\in I$, and $\Uf$ an ultrafilter on $I$,
   we will show that  $\Pi_\Uf\Mt_i\notin \Nmatr(\der)$. 
    For $1\leq j\leq k$ let
$X_j=\{i\in I:\Gamma_j\der_{i}\Delta_j$\}.
As $\Mt_i\notin \Nmatr(\der)$ for every $i\in I$, we have that
  $X_1\cap \ldots  \cap X_k=\emptyset\notin\Uf$.
  Hence,  
   there exists 
$1\leq \ell\leq k$ such that $X_\ell\notin \Uf$, and since $\Uf$ is an ultrafilter, $Y=I\setminus X_\ell=\{i\in I:\Gamma_\ell\not\der_{i}\Delta_\ell\}\in \Uf$. 
For each $i\in I$, pick $v_i\in \Val(\Mt_i)$, in such a way that $v_i[\Gamma_\ell]\subseteq D_i$
 and $v_i[\Delta_\ell]\cap D_i=\emptyset$ whenever $i\in Y$.
Using Proposition~\ref{valsUP},
there exists $v\in \Val(\Pi_\Uf\Mt_i)$ such that $v(\varphi)=[s_\varphi]_\Uf$ with $\pi_i(s_\varphi)=v_i(\varphi)$ for every $\varphi\in Fm$.

It follows that $Y\subseteq\{i\in I:v_i[\Gamma_\ell]\subseteq D_i\}\subseteq\{i\in I:v_i(\varphi)\in D_i\}\in \Uf$ for each $\varphi\in\Gamma_\ell$, and thus 
$v[\Gamma_\ell]\subseteq (\Pi_{i\in I}D_i)_{/\!\equiv_\Uf}$.

Also, $X_\ell\supseteq\{i\in I:v_i[\Delta_\ell]\cap D_i\neq\emptyset\} \notin \Uf$, and therefore $\{i\in I:v_i[\Delta_\ell]\cap D_i=\emptyset\}\subseteq\{i\in I:v_i(\psi)\notin D_i\}\in\Uf$ for each $\psi\in\Delta_\ell$. We get that 
$\{i\in I:v_i(\psi)\in D_i\}\notin\Uf$ for each $\psi\in\Delta_\ell$, and thus 
$v[\Delta_\ell]\cap (\Pi_{i\in I}D_i)_{/\!\equiv_\Uf}=\emptyset$.

We conclude that
$\Gamma_\ell\not\der_{\Pi_\Uf\Mt_i}\Delta_\ell$ and so $\Pi_\Uf\Mt_i\notin \Nmatr(\der)$. 
\smallskip

Reciprocally, let us assume that $\der$ is not finitely based. For $i<\omega$, let $\Theta_i=\{\varphi\in Fm: \var(\varphi)\subseteq\{p_j:j\leq i\}, \mathsf{depth}(\varphi)\leq i\}$, and consider the sequence of logics $\der_i{=}\der_{R_i}$ with $R_i=\{\frac{\Gamma}{\Delta}:\Gamma,\Delta\subseteq \Theta_i, \Gamma\der \Delta\}$. Note that $\Theta_i$ is always finite, and thus each $\der_i$ is finitely based. Note also that  $\der_i{\subseteq} \der_j$ for $i\leq j$ and, since we are assuming $\der$ is compact, that $\der{=}\der_R$ with $R=\bigcup_{i<\omega}R_i$. 

Since $\der$
is not finitely based we have that each $\der_i{\subsetneq} \der$,
and we can pick an Nmatrix $\Mt_i=\tuple{A_i,\cdot_i,D_i}\in\Nmatr(\der_i)\setminus \Nmatr(\der)$. 
  We consider the ultrafilter $\Uf$ generated by the cofinite subsets of $\nats$
  and show that $\Pi_\Uf\Mt_i\in \Nmatr(\der)$.
  
  If $\Gamma\der \Delta$, by compactness,
  there are finite $\Gamma_0\subseteq \Gamma$ and 
  $\Delta_0\subseteq \Delta$ such that $\Gamma_0\der \Delta_0$.
  Thus, there exists $k<\omega$ such that $\Gamma_0\cup \Delta_0\subseteq \Theta_k$,
  and hence  
  $\Gamma_0\der_{i}\Delta_0$ for
  $i\geq k$.
  Clearly, using cofiniteness, 
  $\{i<\omega:k\leq i\}\subseteq X=\{i<\omega:\Gamma_0\der_{i}\Delta_0\}\in \Uf$.
 Let $v\in \Val(\Pi_\Uf\Mt_i)$ such that
  $v[\Gamma_0]\subseteq (\Pi_{i\in I} D_i)_{/\equiv_{\Uf}}$. Recall from Proposition~\ref{prop:decomp} that there 
 is a family of valuations $v_i\in\Val(\Mt_i)$ 
 such that
 for every formula $\varphi$ we have
 $v(\varphi)=[s_\varphi]_\Uf$ where $\pi_i(s_\varphi)=v_i(\varphi)$ for all $i\in I$. We then have $X_\varphi=\{i<\omega:v_i(\varphi)\in D_i\}\in\Uf$ for each $\varphi\in\Gamma_0$, and thus also
 $X\cap\bigcap_{\varphi}X_\varphi\subseteq\{i<\omega:v_i[\Delta_0]\cap D_i\neq\emptyset\}\subseteq
 \{i<\omega:v_i(\psi)\in D_i\}\in\Uf$ for some $\psi\in\Delta_0$. Thus, we have that $v[\Delta_0]\cap(\Pi_{i\in I} D_i)_{/\equiv_{\Uf}}\neq\emptyset$, 
and so $\Gamma_0\der_{\Pi_\Uf\Mt_i}\Delta_0$ and 
$\Gamma\der_{\Pi_\Uf\Mt_i}\Delta$.
 We conclude that
 $\Pi_\Uf\Mt_i\in\Nmatr(\der)$ and the complement of $\Nmatr(\der)$ is not closed under ultraproducts.
 \end{proof}

It is well known that the single-conclusion logic induced by a finite logical matrix $\Mt$ need not be finitely axiomatizable by means of single-conclusion rules~\cite{wronski}. However, Shoesmith and Smiley~\cite{ShoesmithSmiley} have shown that the corresponding multiple-conclusion logic is always finitely based, and therefore the class of $\Nmatr(\der_\Mt)$ and its complement are both closed under ultraproducts. As observed by Czelakowski~\cite{czela1983},  Zygmunt~\cite{Zygmunt} has given a model-theoretic proof of this fact.

In general the above is not true about finite Nmatrices. The paper~\cite{synt} presents a 3-valued Nmatrix $\Mt$ whose logic is not finitely based, and 
by the same  technique used in the proof of Theorem~\ref{charfinbased} one can show 
that the complement of $\Nmatr(\der_\Mt)$ is not closed under 
ultraproducts. It is also shown in~\cite{synt} that, if $\Mt$ is {monadic}, then $\der_\Mt$ is always finitely based, and we can even effectively extract from $\Mt$  an analytic multiple-conclusion calculus for the logic.

\subsection{Characterizing $\Nmatr(\der)$ when  $\der$ is compact}

We begin with a few technical lemmas. Given a set $I$, consider two families of $\Sigma$-Nmatrices $\Mt_i=\tuple{A_i,\cdot_{\Mt_i},D_i}$ and $\Mt'_i=\tuple{A'_i,\cdot_{\Mt'_i},D'_i}$ for $i\in I$. Having fixed an ultrafilter $\Uf$ on $I$, let $\equiv$ and $\equiv'$ denote the equivalence relations (both defined as $\equiv_\Uf$) induced on $\Pi_{i\in I}\Mt_i$ and on  $\Pi_{i\in I}\Mt'_i$, respectively.
\begin{lemma}\label{lem:ultracover} 
    Given a family of strict homomorphisms $f_i:\Mt_i\to \Mt'_i$ for $i\in I$,
    we have that
    $f=\Pi_{i\in I}f_i:\Pi_{i\in I}\Mt_i\to \Pi_{i\in I}\Mt'_i$ and  $f_\Uf:\Pi_\Uf\Mt_i\to \Pi_\Uf\Mt'_i$
    given by
    $f_\Uf([s]_\equiv)=[f(t)]_{\equiv'}$ constitute strict homomorphisms, which are covering when 
    $f_i$ is covering for every $i\in I$.
\end{lemma}
\begin{proof}
It is straightforward to check that $f=\Pi_{i\in I}f_i$ constitutes a strict homomorphism, 
that $f_\Uf$ is well defined, and that $f:\Pi_\Uf\Mt_i\to \Pi_\Uf\Mt'_i$ is also
a strict homomorphism. Let us show that $f$ and $f_\Uf$ are covering, when assuming that all the $f_i$ are. 
For convenience, let  $\Mt_a=\Pi_{i\in I}\Mt_i$ and $\Mt_b=\Pi_{i\in I}\Mt'_i$ (with $a,b\notin I$).

To see that $f$ is covering, let $t,t_1,\ldots,t_k\in \Pi_{i\in I}A_i$ with $t\in \conn_{\Mt_b}(t_1,\dots,t_k)$. Thus we have, for each $i\in I$, $\pi_i(t)\in \conn_{\Mt'_i}(\pi_i(t_1),\ldots,\pi_i(t_k))$. Since each $f_i$ is covering there 
exist $x^i,x^i_1,\ldots x^i_k\in A_i$ such that $x^i\in \conn_{\Mt_i}(x^i_1,\ldots,x^i_k)$, $f_i(x^i)=\pi_i(t)$ and
$f_i(x^i_j)=\pi_i(t_j)$ for $1\leq j\leq k$.
Hence, we have $s,s_1,\ldots,s_k\in \Pi_{i\in I}A_i$ such that, for each $i\in I$, $\pi_i(s)=x^i$ and $\pi_i(s_j)=x^i_j$ for $1\leq j\leq k$,
and $s\in \conn_{\Mt_a}(s_1,\ldots,s_k)$, $f(s)=t$ and $f(s_j)=t_j$ for $1\leq j\leq k$.

To see that $f_\Uf$ is covering,
let $[t]_{\equiv'}\in \conn_{\equiv'}([t_1]_{\equiv'},\dots,[t_k]_{\equiv'})$. Using Lemma~\ref{lem:up},
we know that $X=\{i\in I: \pi_i(t)\in \conn_{\Mt'_i}(\pi_i(t_1),\ldots,\pi_i(t_k))\}\in \Uf$.
Proceeding as in the above paragraph for 
each $i\in X$, we get $s,s_1,\ldots,s_k\in \Pi_{i\in I}A_i$ such that
   $X\subseteq  \{i\in I:f_i(\pi_i(s))=\pi_i(t)\} \in \Uf$, 
   $X\subseteq\{i\in I:f_i(\pi_i(s_j))=\pi_i(t_j)\} \in \Uf$ for every $1\leq j\leq k$, and thus
 $X\subseteq  \{i\in I:\pi_i(s)\in \conn_{\Mt_i}(\pi_i(s_1),\ldots,\pi_i(s_k))\} \in \Uf$. 
Thus, we have 
$f([s]_\equiv)=[t]_{\equiv'}$, $f([s_j]_{\equiv})=[t_j]_{\equiv'}$ for $1\leq j\leq k$ and
$[s]_\equiv\in \conn_{\equiv}([s_1]_\equiv,\dots,[s_k]_\equiv)$.
\end{proof}

\begin{lemma}\label{lemswap}
Let $\cM$ be a class of $\Sigma$-Nmatrices. If $\der{=}\der_{\cM}$ is compact, then 
 $\SubN(\Up(\Quo(\Rexp({\SubN}(\cM)))))\subseteq \Quo(\Rexp({\SubN}(\Up(\cM))))$.

\end{lemma}
\begin{proof}
If $\Mt=\tuple{A,\cdot_\Mt,D}\in \SubN(\Up(\Quo(\Rexp({\SubN}(\cM)))))$ then
$\Mt\in \SubN(\{\Pi_\Uf\Mt_i\})$ for some family of $\Sigma$-Nmatrices $\Mt_i\in\cM$ for $i\in I$, and some ultrafilter $\Uf$ on $I$, and for each $i\in I$ there are:
\begin{itemize}
      
\item covering strict homomorphisms $g_i:\Mt'_i\to \Mt_i$ 
with $\Mt_i=\tuple{A_i,\cdot_i,D_i}$ and $\Mt'_i=\tuple{A'_i,\cdot'_i,D'_i}$
such that 
$\Mt_i=g_i[\Mt'_i]$ and $\der{\subseteq} \der_{\Mt_i}$,

\item covering strict homomorphisms $h_i:\Mt'_i\to \Mt^-_i$ such that $\Mt^-_i=h_i[\Mt'_i]$, and

\item  $\Mt^-_i\in \SubN(\Mt^+_i)$ with $\Mt^+_i\in \cM$.

\end{itemize}

By Lemma~\ref{lem:ultracover} there are covering strict homomorphisms $g_\Uf:\Pi_\Uf \Mt'_i\to \Pi_\Uf \Mt_i$ and $h_\Uf:\Pi_\Uf \Mt'_i\to \Pi_\Uf \Mt^-_i$.
Let $\Mt_a\in \SubN(\Pi_\Uf \Mt'_i)$ be the preimage of $\Mt\in \SubN(\Pi_\Uf \Mt_i)$ by $g_\Uf$, and 
$\Mt_b\in \SubN(\Pi_\Uf \Mt^-_i)$ be the image of $\Mt_a$ by $h_\Uf$. So $\Mt=g_\Uf[\Mt_a]$ and $\Mt_b=h_\Uf[\Mt_a]$.
Further, we have that
$\Mt_b\in \SubN(\Pi_\Uf \Mt^-_i)\subseteq\SubN(\Pi_\Uf \Mt^+_i)\subseteq\SubN(\Up(\cM))$ and therefore
 $\Mt_a\in \Rexp(\{\Mt_b\})\subseteq \Rexp(\SubN(\Up(\cM)))$.
Therefore, as $\der$ is compact, we can use Proposition~\ref{soundup} along with Lemma~\ref{lem:rexpstronger} to get that $\der_\Mt{\subseteq} \der$,
and conclude that 
 $\Mt\in \Quo(\{\Mt_a\})\subseteq \Quo(\Rexp(\SubN(\Up(\cM))))$. 
   %
\end{proof}

The following characterization is now natural, and generalizes 
\cite[Thm.~I.7]{czela1983} (which is~Thm.~3.10 in Font and Jansana's
chapter included in the present book).

\begin{theorem}\label{soundmodelscompact}
Let $\cM$ be a class of $\Sigma$-Nmatrices. If $\der{=}\der_{\cM}$ is compact, then
 $\Nmatr(\der)=\Quo(\Rexp({\SubN}(\Up(\cM))))$.
\end{theorem}
\begin{proof}
From Lemma~\ref{valrexp} we know that 
 $\cL_\der\subseteq \Rexp({\SubN}(\cM))$,
 and by Proposition~\ref{breakintoden}
 that
every countable $\der$-sound Nmatrix is 
in $\cL^{\mathsf{quo}}_\der=\Quo(\cL_\der)\subseteq \Quo(\Rexp({\SubN}(\cM)))$.
By
Proposition~\ref{prop:denum}
 we conclude that $\Nmatr(\der_\cM)=\SubN(\Up(\Quo(\Rexp({\SubN}(\cM)))))$. Finally, the result follows by Lemma~\ref{lemswap}.
\end{proof}
Note that, in contrast to what happens in the case of ordinary logical matrices, in our setting 
we have $\Quo\circ\Rexp\neq \Rexp\circ\Quo$, as can be noticed by an inspection of Example~\ref{exquo}. Thus we cannot generate all sound models of $\der$ using exactly the sequence of operators employed in~\cite{czela1983}.

 \section{Conclusions and futher work}\label{sec:conc}

In the previous sections we have  generalized the results of  Czelakowski~\cite{czela1983} about multiple-conclusion logics to the setting of Nmatrix semantics,  highlighting  the main   novelties  and 
 challenges that arise when extending the 
algebraic logic techniques in order to deal with  non-determinism.
Let us conclude by listing a few  possible opportunities for future investigation suggested by the 
crucial differences brought by non-determinism.  
\smallskip

{
As is well known, many of the most prominent logics -- the classical, the intuitionistic,
several  many-valued and modal systems -- 
are \emph{algebraizable}
in the sense of Blok and Pigozzi~\cite{blokpig}. For a single-conclusion logic $\mathcal{L}$, this entails that its consequence relation
can be faithfully interpreted into the relative equational consequence relation of a uniquely determined class 
$\mathsf{alg}(\mathcal{L})$
of universal algebras,
known as the \emph{equivalent algebraic semantics} of $\mathcal{L}$. The theory of algebraizability guarantees that 
$\mathsf{alg}(\mathcal{L})$ will be axiomatized by means of \emph{generalized quasi-equations}, that is,
 first-order conditionals
having  a possibly infinite conjunction of equations as premiss and a single equation as conclusion;
further mild assumptions, which virtually cover all the non \emph{ad hoc} examples in the literature, ensure that 
the conjunction in the premiss of each
first-order formula
may be taken to be finite: thus  $\mathsf{alg}(\mathcal{L})$
will be a \emph{quasi-variety} of algebras. From the perspective of algebraic logic, 
the study of algebraizable single-conclusion logics may in this sense 
be reduced to the universal algebraic study of (generalized) quasi-varieties. 

In a  multiple-conclusion setting we may ask ourselves, in the first place,
whether and how one could single out a fruitful notion of an algebraizable multiple-conclusion logic;
and, secondly,  
which  class of (multi)algebras
one should look at as a potential candidate for 
the algebraic
counterpart of a  multiple-conclusion logic. 
It is an easy guess that the first-order sentences required for axiomatizing
such a class  
will be conditionals having a conjunction of equations as a premiss
and a disjunction as a conclusion.
We give a more detailed suggestion
in this direction in the following example.

\begin{example}
Given a class $\mathcal{A}$ of $\Sigma$-multialgebras, let 
$\mathsf{Eqs}=\{\varphi=\psi:\varphi,\psi\in Fm\}$, and further define $\der_{\mathcal{A}}{\subseteq} \,\wp(\mathsf{Eqs})\times \wp(\mathsf{Eqs})$ as follows:  
$$\{\varphi_i=\psi_i:i\in I\}\der_{\mathcal{A}} \{\varphi_j=\psi_j:j\in J\}$$
whenever 
for every $\mathbf{A}\in  \mathcal{A}$  and for every multialgebra morphism $h:\Fm\to \mathbf{A}$   
$h(\varphi_i)=h(\psi_i)$ for every $i\in I$ implies   $h(\varphi_j)=h(\psi_j)$ for some $j\in J$.

It is easy to verify that $\der_{\mathcal{A}}$ satisfies suitably adapted versions of the properties $(O)$, $(D)$, $(C)$ and $(S)$
in the setting of a `multiple-conclusion logic' that
manipulates equations instead of formulas (i.e.,~a 2-deductive system). Clearly, 
we have that
$R_{\mathsf{eq}} =\{{\mathsf{r_{ref}}},{\mathsf{r_{symm}}},{\mathsf{r_{trans}}}\}\,{\subseteq} \der_{\mathcal{A}}$, where the rules are defined as follows:
$$\frac{}{x=x}~{\mathsf{r_{ref}}}\qquad\frac{x=y}{y=x}~{\mathsf{r_{symm}}}\qquad\frac{x=y\,,\,y=z}{x=z}~{\mathsf{r_{trans}}}$$
Given a multialgebra $\mathbf{A}=\{A,\cdot_\mathbf{A}\}$
such that 
for each $a\in A$ there is $\tilde{a}\in \Sigma^{0}$ with
$\tilde{a}_\mathbf{A}=\{a\}$,
we also have that 
$R_\mathbf{A}=R_{\mathsf{eq}}\cup R_\exists\cup R_{\neq} \cup R_\Sigma\, {\subseteq} \der_{\{\mathbf{A}\}}$
 with the following sets of rules. 
\begin{align*}
     R_\exists&=\left\{\frac{}{\{x=\tilde{a}:a\in A\}}\right\}\\   
     R_{\neq}&= \left\{\frac{x=\tilde{a}\,,\,x=\tilde{b}}{}: a\neq b\in A\right\}\\
R_\Sigma&=\left\{\mathsf{r}_{\conn,a_1,\ldots,a_k}:\conn\in \Sigma^{(k)},a_1,\ldots, a_k\in A\right\}\\ 
 {\mathsf{r}_{\conn,a_1,\ldots,a_k}}&=\frac{\{x_j=\tilde{a}_j:1\leq j\leq k\}}{\{\conn(x_1,\ldots,x_k)=\tilde{b}:b\in \conn_{\mathbb{A}}(a_1,\ldots, a_k)\}}
\end{align*}
It is also not hard to verify that $R_\mathbf{A}$ forms a basis for 
$\der_{\{\mathbf{A}\}}$. 
In particular, when $\mathbf{A}$ is a (deterministic multi)algebra,  congruentiality is derivable.
Indeed, assuming $x_i=y_i$ for each $1\leq i\leq k$, and using the equality rules $R_{\mathsf{eq}}$
 and $R_{\exists}$ 
 we conclude 
that there are $\tilde{a}_{i}\in \Sigma^{0}$ such that
$x_i=y_i=\tilde{a}_{i}$. Then, since $\mathbf{A}$ is deterministic the rule ${\mathsf{r}_{\conn,a_1,\ldots,a_k}}$ has only one conclusion, i.e., it matches the single conclusion rule 
$\frac{\{x_j=\tilde{a}_{j}:1\leq j\leq k\}}{\conn(x_1,\ldots,x_k)=\tilde{b}}$ for $\conn_{\mathbb{A}}(a_1,\ldots, a_k)\}=\{b\}$
and thus we obtain that 
$\conn(x_1,\ldots,x_k)=\conn(y_1,\ldots,y_k)$.

Given a subset of designated elements $D\subseteq A$ and $\Mt=\tuple{A,\cdot_\mathbf{A},D}$
we can recover the logic $\der_{\Mt}$ given by
\begin{center}
 $\Gamma\der_{\Mt}\Delta$\\  iff 

for every function $f:\Gamma\to \{\tilde{a}:a\in D\}$,

 $\{\gamma=f(\gamma):\gamma\in \Gamma\}\der_{\{\mathbf{A}\}} \{\delta=\tilde{a}:\delta\in \Delta,a\in D\}.$ 
 
\end{center}
Thus, under the conditions of this toy example, we can associate to a logic induced by an Nmatrix 
clauses of equations into which the logic can be embedded.
Given an arbitrary logic (whose Nmatrix semantics we do not know), we can then try to associate a class of multi-algebras 
when having something like algebraizability withouth the congruentiality assumption.
\hfill$\triangle$
\end{example}
}

A basis for a logic is usually seen as a Hilbert-style deductive system. 
In this paper, we have not explored the deductive side of multiple-conclusion bases, but 
the hand-waved derivation in the above example could be formalized as a tree-like proof~\cite{ShoesmithSmiley,synt,wollic19}.
Further, the connection between logics and clausal equational logics, also hinted at in the above example, may be an interesting bridge to explore, namely by importing analyticity results in logic to produce symbolic decision procedures for equational reasoning~\cite{synt}.\smallskip

As  mentioned in the Introduction,   we do not currently have a theory of logical congruences for non-deterministic matrices, nor a standard counterpart to what is   the Leibniz operator for ordinary matrices. We have also observed earlier that it is reasonable to anticipate that the role of logical congruences will be played by ``logical equivalence relations'', i.e.,~compatible equivalence relations. 
{
Solutions to the associated problem of what is the right generalization of the notion of reduced model  may lead to natural extensions of Corollary I.14 in~\cite{czela1983} (Thm.~3.11 in Font and Jansana's
chapter included in the present book).
One would also wish to introduce a 
finer notion of homomorphism between two Nmatrices which could guarantee that both define the same logic    
 -- or at least a logic not weaker than some fixed ambient logic possibly given by a finite Nmatrix.
In this respect, a potential difficulty may stem from the following observation: 
while the problem of determining whether two finite deterministic matrices
define the same logic  is known to be decidable (for both the single- and the multiple-conclusion case),  
we  know as well that this is not the case, at least in the single-conclusion case, for finite Nmatrices \cite{FiMaCa22JLC}.  

Another point where the theory of Nmatrices seems to diverge further from that of ordinary matrices is 
in regard to  Corollary I.13 of Czelakowski~\cite{czela1983} stating that the multiple-conclusion logic
defined by a finite deterministic matrix only admits finitely many strengthenings. We know that this result does not
extend to finite Nmatrices: for example, the strengthenings of the logic of the unconstrained 2-valued Nmatrix ($\mathbb{U}_{1,1}$ in Example~\ref{ex2val}) are all the possible logics. 

{
Finally, an obvious direction to look for further generalizations of the results presented here is 
the theory of partial Nmatrices (PNmatrices). 
 In order to obtain
 information on 
the class $\mathsf{PNmatr}(\der)$ of all partial non-deterministic  models of a logic $\der$,
we may start by observing that  
 every PNmatrix is a quotient of a coproduct of its subNmatrices.
The logic induced by a class of Nmatrices is not necessarily the same as that of all its submatrices,
but the logic induced by a class of PNmatrices is the same as that of all its subNmatrices, which can be 
glued back together into the original PNmatrix.

\section*{Acknowledgments}
Carlos Caleiro's and S\'{e}rgio Marcelino's research was done under the scope of project FCT/MCTES through national funds and when applicable co-funded by EU under the project UIDB/50008/2020.

\end{document}